\documentclass[preprint,11pt]{elsarticle}
\makeatletter
\def\ps@pprintTitle{%
 \let\@oddhead\@empty
 \let\@evenhead\@empty
 \def\@oddfoot{\centerline{\thepage}}%
 \let\@evenfoot\@oddfoot}
\makeatother
\usepackage{amssymb,amsmath,amsthm}
\usepackage{epsfig}
\usepackage{color}
\usepackage{makeidx}
\usepackage{bbm}
\usepackage[top=1in, bottom=1in, inner=1in, outer=1in]{geometry}
\usepackage{setspace}

\newtheorem{theorem}{Theorem}
\newtheorem*{thma}{Theorem A}
\newtheorem*{thmb}{Theorem B}
\newtheorem{definition}{Definition}
\newtheorem{lemma}{Lemma}
\newtheorem{proposition}{Proposition}
\newtheorem{corollary}{Corollary}
\newtheorem*{remark}{Remark}
\newcommand*\xbar[1]{%
  \hbox{%
    \vbox{%
      \hrule height 0.5pt 
      \kern0.4ex
      \hbox{%
        \kern-0.15em
        \ensuremath{#1}%
        \kern-0.15em
      }%
    }%
  }%
}

\begin{document}
\begin{frontmatter}

\title{Optimal Survival Strategy for Branching Brownian Motion in a Poissonian Trap Field}

\author[ozyegin]{Mehmet \"{O}z}
\ead{mehmet.oz@ozyegin.edu.tr}
\ead[url]{https://www.ozyegin.edu.tr/en/faculty/mehmetoz}

\author[colorado]{J\'{a}nos Engl\"{a}nder}
\ead{janos.englander@colorado.edu}
\ead[url]{http://www.colorado.edu/math/janos-englander}

\address[ozyegin]{Department of Natural and Mathematical Sciences, Faculty of Engineering, \"{O}zye\u{g}in University, Istanbul, Turkey}
\address[colorado]{Department of Mathematics, University of Colorado at
Boulder, Boulder, CO-80309, USA}

\begin{abstract}
We study a branching Brownian motion $Z$ with a generic branching law, evolving in $\mathbb{R}^d$, where a field of Poissonian traps is present. Each trap is a ball with constant radius. We focus on two cases of Poissonian fields: a uniform field and a radially decaying field. Using classical results on the convergence of the speed of branching Brownian motion, we establish precise results on the population size of $Z$, given that it avoids the trap field, while staying alive up to time $t$. The results are stated so that each gives an `optimal survival strategy' for $Z$. As corollaries of the results concerning the population size, we prove several other optimal survival strategies concerning the range of $Z$, and the size and position of clearings in $\mathbb{R}^d$. We also prove a  result about the hitting time of a single trap by a branching system (Lemma \ref{lemma1}), which may be useful in a completely generic setting too.

{\it Inter alia}, we answer some open problems raised in [{\em Mark.\ Proc.\ Rel.\ Fields}\ \textbf{9} (2003), 363 -- 389].
\end{abstract}

\vspace{3mm}

\begin{keyword}
Branching Brownian motion \sep Poissonian traps \sep Random environment  \sep Hard
obstacles \sep  Optimal survival strategy
\vspace{3mm}
\MSC[2010] 60J80 \sep 60K37 \sep 60F10
\end{keyword}

\end{frontmatter}

\pagestyle{myheadings}
\markright{OPTIMAL SURVIVAL FOR BBM IN A POISSONIAN TRAP FIELD\hfill}

\section{Introduction}\label{intro}

Branching Brownian motion (BBM) in Poissonian trap fields has been studied recently in \cite{E2000,E2003,E2008,OC2013,OCE2017}. The most classical problem on this model is the large time asymptotics of the survival probability of the BBM, where one defines survival up to time $t$ to be the event that none of the particles of the BBM has hit the trap field until that time. Another classical problem is that of the optimal survival strategies: How must have the system behaved (what strategy must it have followed) given that it has avoided the traps up to time $t$? In this work, we study the optimal survival for a BBM that evolves in $\mathbb{R}^d$, where a Poissonian trap field is present. Our focus is on the population size. Conditioned on survival among traps, we expect the system to suppress branching and produce fewer particles than it otherwise would (had it not been conditioned on survival). Here, we quantify how much the branching would be suppressed. Investigation of this problem leads us to proving an important lemma of independent interest, which provides an upper bound that is valid for large $t$ on the survival probability of a BBM in a large class of random (or deterministic) trap fields, not restricted to Poissonian fields.

Next, we describe the two sources of randomness.

\medskip
{\bf 1. Branching Brownian motion:} Let $Z=(Z(t))_{t\geq 0}$ be a $d$-dimensional BBM with initial distribution $\delta_0$, branching rate $\beta>0$, and offspring distribution $(p_k)_{k\in\mathbb{N}_0}$, where $t$ represents time. The process starts with a single particle at the origin, which performs a Brownian motion in $\mathbb{R}^d$ for a random time which is
distributed exponentially with constant parameter $\beta$. Then, the particle dies and simultaneously gives birth to a random number of particles distributed according to the offspring distribution $(p_k)_{k\in\mathbb{N}_0}$, where $p_k\geq 0$ for each $k\in\mathbb{N}_0$ and $\sum_{k=0}^\infty p_k=1$. Similarly, each offspring particle repeats the same procedure independently of all others and the parent, starting from the position of her parent. In this way, one obtains a measure-valued Markov process $Z=(Z(t))_{t\geq0}$, where for each $t\geq 0$, $Z(t)$ can be viewed as a particle configuration on $\mathbb{R}^d$. By assumption, $Z(0)=\delta_0$. Define the process $|Z|=(|Z(t)|)_{t\geq0}$, where $|Z(t)|$ represents the population size of $Z$ at time $t$. The number of particles in generation $n$ of $|Z|$ is a Galton-Watson process $N=(N(n))_{n\in\mathbb{N}}$ with offspring distribution $(p_k)_{k\in\mathbb{N}}$. The initial particle present at $t=0$ constitutes the $0$th generation, the offspring of the initial particle constitute the $1$st generation, and so forth. We denote the extinction time of the process $|Z|$ by $\tau$, which is formally defined as $\tau=\inf\left\{t\geq 0:|Z(t)|=0\right\}$, where we set $\inf\emptyset=\infty$. We then denote the event of extinction of the process $|Z|$ by $\mathcal{E}$, and formally write $\mathcal{E}=\left\{\tau<\infty \right\}$. We use the term \textit{non-extinction} for the event $\mathcal{E}^c$. In this work, $P$ and $E$ denote respectively the probability law and corresponding expectation for the BBM. Finally, for $t\geq 0$, let
\begin{equation} R(t):=\underset{s \in\text{[0,t]}}{\bigcup} \text{supp}(Z(s)) \nonumber
\end{equation}
be the range of $Z$ up to time $t$.

\medskip
{\bf 2. Trap field:} The branching Brownian motion is assumed to live in $\mathbb{R}^d$, to which a `random trap field' is attached. That is,
besides the process $Z$, on some additional space $(\Omega,\mathbb{P})$  (with expectation $\mathbb{E}$), we also consider a $d$-dimensional Poisson random measure $\Pi$,  with a boundedly finite mean measure $\nu$.
By a `trap' associated to a trap point at $x\in\mathbb{R}^d$, we mean a closed ball of fixed radius $r>0$ centered at $x$;
by a (random) `trap field', we mean the set
\[ K:=\underset{x_i \in \text{supp}(\Pi)}{\bigcup} \bar{B}(x_i,r), \]
where $\bar{B}(x,r)$ denotes the closed ball centered at $x\in\mathbb{R^d}$ with radius $r$.  By a `clearing', we mean a region in $\mathbb{R}^d$ that is free of traps, that is disjoint from $K$.  

\begin{definition}[Survival] \rm We define $T:=\inf\left\{t\geq 0:R(t)\cap K\neq\emptyset\right\}$ to be the `first trapping time' of the BBM. By `survival up to time $t>0$,' we mean the event $\mathsf{S_t}:=\left\{T>t\right\}\cap \mathcal{E}^c$, which, in case of $p_0=0$, reduces of course to $\left\{T>t\right\}$. This is the event that neither the internal branching mechanism nor the trap field has killed the process by $t$; thus, {\bf survival is  a subset of non-extinction}, according to our terminology.
\end{definition}

For the two types of Poissonian fields that we consider here, the following asymptotics for the annealed trap-avoiding probabilities have been derived in \cite{O2016} and \cite{OCE2017}, respectively. Let $\text{d}x$ denote the Lebesgue measure.    

\begin{thma}[Survival asymptotics in a uniform field; $d\geq 2$, \cite{O2016}] \label{theorem01}
Let $\mu>1$. Suppose that $\text{d}\nu/\text{d}x=v$, $v>0$. Then, for $d\geq 2$,
\begin{equation}  \underset{t\rightarrow\infty}{\lim}\:\frac{1}{t}\log(\mathbb{E}\times P)\left(T>t\mid \mathcal{E}^c\right)=-\beta \alpha. \label{001}
\end{equation}
\end{thma}

Now consider a trap field where the intensity is radially decaying as
\begin{equation} \frac{\text{d}\nu}{\text{d}x} \sim \frac{l}{|x|^{d-1}}, \quad
|x|\rightarrow\infty, \quad l>0. \label{eq00}
\end{equation}

\begin{thmb}[Survival asymptotics in a radially decaying field; $d\geq 1$, \cite{OCE2017}] \label{theorem02}
Let $\mu>1$. Suppose that $\text{d}\nu/\text{d}x$ exists, is continuous on $\mathbb{R}^d$, and satisfies (\ref{eq00}). For $r,b \geq 0$, define
\begin{equation} g_d(r,b)=\int_{B(0,r)} \frac{\text{d}x}{|x+be|^{d-1}}, \label{003}
\end{equation}
where $e=(1,0,\ldots,0)$ is the unit vector in the direction of the first coordinate. Then, for $d\geq 1$,
\begin{equation} \underset{t\rightarrow\infty}{\lim}\:\frac{1}{t}\log(\mathbb{E}\times P)\left(T>t\mid \mathcal{E}^c\right)=-I(l,f,\beta,d), \label{002} 
\end{equation}
where
\begin{equation} I(l,f,\beta,d)=\underset{\eta \in [0,1],c \in [0,\sqrt{2\beta}]}{\text{min}} \left\{\beta\alpha\eta+\frac{c^2}{2\eta}+l g_d(\sqrt{2\beta m}(1-\eta),c)\right\}. \label{I}
\end{equation}
(For $\eta=0,c=0$, set $c^2/2\eta=0$, and for $\eta=0,c>0$, set $c^2/2\eta=\infty$.)
\end{thmb} 

The constant $\alpha$ appearing in \eqref{001} and \eqref{I} above is defined as $\alpha:=1-f'(q)$, where $q:=P(\mathcal{E})$. When $p_0=0$, it is clear that $P(\mathcal{E})=0$ so that the conditioning on $\mathcal{E}^c$ is redundant, and since $p_1=0$ by assumption, $\alpha=1$. 

In this paper, the probability measure of interest is $(\mathbb{E}\times P)(\:\cdot \mid \mathsf{S_t})$, the annealed (averaged) probability conditioned on survival  up to $t$. 

\begin{definition}[Optimal survival strategy]\rm
By an (annealed) `optimal survival strategy,' we mean a collection of events $\left\{A_t\right\}_{t>0}$ indexed by time, such that 
\begin{equation} \underset{t\rightarrow\infty}{\lim}(\mathbb{E}\times P)\left(A_t \mid \mathsf{S_t}\right)=1. \nonumber
\end{equation}
\end{definition}
We look for optimal survival strategies concerning mainly the population size. 

\smallskip
The problem of trap-avoiding asymptotics for BBM among Poissonian traps has been first studied by Engl\"{a}nder in \cite{E2000}, where a uniform field was considered in $d\geq 2$. Then, in search for an extension to the case $d=1$, Engl\"{a}nder and den Hollander \cite{E2003} studied the more interesting case where the trap intensity was radially decaying as given in (\ref{eq00}). In both \cite{E2000} and \cite{E2003}, the main result was the exponential asymptotic decay rate of the annealed survival probability as $t\rightarrow\infty$, and the branching was taken to be strictly dyadic, i.e., $p_2=1$. In addition, in \cite{E2003}, optimal survival strategies of the type we consider here were proved (see Theorem 1.3 (i)-(iv) therein). Part of the work in this paper could be regarded as a refinement and generalization of the corresponding work in \cite{E2003}.         

In \cite{E2000}, optimal survival strategies were not studied. In Theorem~\ref{theorem1}, we consider a uniform field in $\mathbb {R}^d,\ d\ge 2$, as in \cite{E2000}, and prove that conditioned on survival up to time $t$, for any $0<\widehat\varepsilon<1$, with overwhelming probability, there is only $1$ particle present at time $t(1-\widehat\varepsilon)$ as $t\rightarrow\infty$, which means complete suppression of branching occurs with overwhelming probability. In Theorem 1.3 (iii) in \cite{E2003}, where the setting was that of the radially decaying trap intensity in (\ref{eq00}), it was shown that conditioned on survival up to time $t$, the population size at time $(\eta^*-\widehat\varepsilon)t$ is at most $\lfloor t^d \rfloor$ for large $t$ with overwhelming probability. (The constant $\eta^*$ appears as one of the minimizers in \eqref{I}.) Here, we improve this bound to just $1$ in Theorem~\ref{theorem2}. 

The reason the decay rate given in (\ref{eq00}) is the `interesting' one is that it is in fact the `borderline' one. This is explained in Theorem 1.3 in \cite{E2003}, which describes the optimal survival strategy, as it depends on the `fine tuning constant' $\ell$ (we use  $l$ instead of $\ell$ in the present paper). Namely, it was shown that 
\begin{itemize}
\item[--]
In the low intensity regime $\ell<\ell_{cr}$, the system clears a ball of radius $\sqrt{2\beta}\,t$ from traps, and until time $t$ stays inside this ball and branches at rate $\beta$.

\item[--]
In the high intensity regime $\ell>\ell_{cr}$:
\begin{itemize}
\item[$d=1$:]
The system clears an $o(t)$-ball (i.e., a ball with radius $>a$ but $\ll t$), and until time $t$ suppresses the branching (i.e., produces a polynomial number of particles) and stays inside this ball.

\item[$d\geq 2$:]
The system clears a ball of radius $\sqrt{2\beta}\,(1-\eta^*)t$ around a point at distance $c^*t$ from the origin, suppresses the branching until time $\eta^*t$, and during the remaining time $(1-\eta^*)t$  branches at rate $\beta$.

\end{itemize}
\end{itemize}
(See Theorem 1.3 in \cite{E2003} for the precise statements.)

Hence, the decay considered is indeed the `borderline' one, where the behavior of the system depends only on the constant $\ell$, and exhibits a change of behavior at the crossover. If one considers a larger (smaller) decay order, the optimal strategy will simply follow the one exhibited when the decay is as in (\ref{eq00}) and $\ell>\ell_{cr}$ ($\ell<\ell_{cr}$); although if the decay order is very large, then $\eta^*=1$  (complete suppression of branching) may occur even for $d\ge 2$, while $0<\eta^*<1$ is always the case in the high intensity regime studied in \cite{E2003}.

When the BBM is supercritical and $p_0>0$, we have to take into account that extinction for the underlying Galton-Watson process has positive probability, and hence condition the process on non-extinction. In this case, the particles are grouped into those with infinite or finite line of descent, so-called `skeleton' and `doomed' particles, respectively, and in this way a `skeleton decomposition' is performed to analyze the problem. In \cite{OC2013} and \cite{O2016}, the work in \cite{E2000} was extended to a BBM with a general offspring distribution, where the possibility of $p_0>0$ was allowed. Likewise in \cite{OCE2017}, a general offspring distribution is considered for the BBM, and the work in \cite{E2003} on the radially decaying trap field is extended to cover the case $p_0>0$. Here we allow for $p_0>0$, and extend Theorem~\ref{theorem1} and Theorem~\ref{theorem2} to Theorem~\ref{theorem3} and Theorem~\ref{theorem4}, respectively, to obtain optimal survival strategies on the population size of both the skeleton and doomed particles.   

In the final section, we use our optimal survival results on the population size to prove several others of different types in the same spirit as in \cite[Thm.1.3]{E2003}, concerning the range of the BBM, and the size and position of trap-free regions (i.e., clearings) in $\mathbb{R}^d$.

We refer the reader to \cite{E2007} for a survey on the topic of BBM among Poissonian traps, and to \cite{E2008,LV2009} for various related problems. Analogous questions in the discrete setting could also be asked, where the continuum $\mathbb{R}^d$ is replaced by the integer lattice $\mathbb{Z}^d$, and the BBM is replaced by the branching random walk. In \cite{ADS2016}, a random walk among a randomly moving field of traps on $\mathbb{Z}^d$ was studied, and it was shown that conditioned on survival up to time $t$, the random walk is subdiffusive. In the discrete setting, we note that the survival asymptotics of the random walk was studied earlier in \cite{R2012} for both the annealed and quenched cases.    

The organization of the paper is as follows. In Section~\ref{section2}, we state the main results. Section~\ref{section3} is devoted to the central lemma of this work, on which the main results are built. This lemma is a general result that applies to a much broader class of trap fields on $\mathbb{R}^d$ than the ones considered in this work. In Section~\ref{section4}, we give the proofs of the main results. In Section~\ref{section5}, we state and prove a lemma of independent interest about the decomposition of a supercritical continuous-time branching process, which is used in the subsequent section to generalize the main results to the case $p_0>0$.  In Section~\ref{section6}, we extend the main results to the case $p_0>0$. The first six sections study the optimal survival strategies on the population size of the branching system. Finally, in Section~\ref{section7}, we provide optimal survival results on the range of the branching system, and the size and position of the clearings in $\mathbb{R}^d$, as corollaries of the results on population size.

\section{Results}\label{section2}

Our main results will be stated so that each gives an optimal survival strategy. Let us now introduce further notation in order to state the results. Let $f$ be the probability generating function (p.g.f.) of the offspring distribution and $\mu$ be the mean number of offspring:
\[f(s):=\sum_{j=0}^{\infty} p_j s^j;\qquad \mu:=\sum_{j=0}^{\infty}j p_j,
\]
and define
\[m:=\mu-1.
\] 
Note the significance of $m$: it is the net average growth per particle since by assumption a particle dies at the moment it gives birth to offspring. 

Throughout this work, we assume that $\mu<\infty$, and without loss of generality that $p_1=0$ (as nonzero $p_1$ can be absorbed into the branching rate $\beta$). Also, from elementary theory of branching processes (see for example \cite{AH1983, AN1972}), recall the following fact: $P(\mathcal{E})=1\Leftrightarrow \mu\leq 1$. Processes for which $\mu>1$ are called supercritical. It is clear that if $p_0=0$, then $P(\mathcal{E})=0$.

The following two theorems constitute the main results of this paper. They both give the population size of the branching system on the condition of survival among traps, and hence quantify how much the branching would be suppressed given survival. The setting in the first theorem is a uniform trap field, whereas the second one is concerned with a radially decaying field. Lemma~\ref{lemma1} in Section~\ref{section3} is central in the proof of both theorems. One should keep in mind that without the conditioning on survival, the expected population size of a `free' BBM at time $t$ is $\exp[\beta m t]$ for $t\geq 0$. Moreover, we have the limit theorem saying that $\underset{t\rightarrow\infty}{\lim}|Z(t)|e^{-\beta m t}$ exists almost surely (see for example \cite[Thm.III.7.1]{AN1972}). 
  
As before, $\text{d}x$ denotes the Lebesgue measure.	

\begin{theorem}[Survival in a uniform field; $d\geq 2$] \label{theorem1}
Let $p_0=0$. Suppose that $\text{d}\nu/\text{d}x=v$, $v>0$. Then for $d\geq 2$ and $0<\widehat\varepsilon<1$,
\begin{equation} \underset{t\rightarrow\infty}{\lim}(\mathbb{E}\times P)\left(|Z((1-\widehat\varepsilon)t)|=1\mid \mathsf{S_t}\right)=1. \nonumber
\end{equation}
\end{theorem}

\begin{remark} Theorem~\ref{theorem1} says that for large $t$, conditioned on survival up to time $t$, with overwhelming probability, the population size at the earlier time $(1-\widehat\varepsilon)t$ is $1$. In other words, with overwhelming probability the population doesn't grow at all up to time $(1-\widehat\varepsilon)t$; branching is completely suppressed. We stress that this is not an almost sure pathwise statement, so there could be realizations where the population grows. 
\end{remark}

The following theorem has the setting of a trap field where the intensity is radially decaying as in \eqref{eq00}:
\begin{equation} \frac{\text{d}\nu}{\text{d}x} \sim \frac{l}{|x|^{d-1}}, \quad
|x|\rightarrow\infty, \quad l>0. \nonumber
\end{equation}
In this case, as emphasized in the introduction, there is a critical intensity $l_{cr}$ at which the switching of regime occurs. The survival strategy of the system depends on whether $l$ is above or below this critical intensity, and in particular for $l>l_{cr}$, the system suppresses the branching until time $\eta^*t$, where $\eta^*$ is one of the minimizers in \eqref{I}. For a definition of and a formula for $l_{cr}$, and for details on $\eta^*$, we refer the reader to \cite[Thm.2]{OCE2017} and its proof. Here, we only note that $0<\eta^*<1$ when $d\geq 2$, and $\eta^*=1$ when $d=1$. Also, it is clear that when $d=1$, the trap intensity in (\ref{eq00}) gives a uniform field as a special case, hence covering the missing case of $d=1$ in Theorem~\ref{theorem1}. 

The next result answers some of the open problems raised in Section 1.3 in \cite{E2003}.

\begin{theorem}[Survival in a radially decaying field; $d\geq 1$] \label{theorem2}
Let $p_0=0$. Suppose that $\text{d}\nu/\text{d}x$ exists, is continuous on $\mathbb{R}^d$, and satisfies (\ref{eq00}). Let $l_{cr}$ be the constant in the critical trap intensity. Then, in the high-intensity regime $l>l_{cr}$, for $d\geq 1$ and $0<\widehat\varepsilon<\eta^*$,
\begin{equation} \underset{t\rightarrow\infty}{\lim}(\mathbb{E}\times P)\left(|Z((\eta^*-\widehat\varepsilon)t)|=1\mid \mathsf{S_t}\right)=1. \nonumber
\end{equation}
\end{theorem}

\begin{remark}
Optimal survival strategies about the population size arise when the branching is suppressed for at least part of the time interval in question in order to realize the event of survival. Therefore, the strategy in the theorem above applies only when $l>l_{cr}$, where the branching is suppressed in the time interval $[0,\eta^* t]$. When $l<l_{cr}$, the system undergoes `free' branching.
\end{remark}

\section{A trap in a subcritical ball}\label{section3}

In this section, we state and prove the central lemma of this work, on which the main results are built. The following lemma is of independent interest, because it applies to a much more general class of trap fields (random or deterministic) on $\mathbb{R}^d$ as opposed to only Poissonian fields.  

\begin{lemma}[Survival among traps in a subcritical ball]\label{lemma1}
Let $p_0=0=p_1$. Let $0<\varepsilon<1$ and define $\rho_t=\sqrt{2\beta m}(1-\varepsilon)t$. Suppose that  {\rm supp}$(\Pi)\cap \bar{B}(0,\rho_t)\neq \emptyset$. Then, the probability that the BBM avoids the trap field up to time $t$ satisfies the following asymptotics:
\begin{equation}\underset{t\rightarrow\infty}{\limsup}\,\frac{1}{t}\log P\left(\mathsf{S_t}\right)\leq-\beta\varepsilon\left(\sqrt{m^2+m}-m\right). \nonumber
\end{equation}
\end{lemma}

\begin{remark} We call $B(0,\rho_t)$, where $\rho_t=\rho_t(\varepsilon)$, a \textit{subcritical ball} since the `speed' of a BBM is equal to $\sqrt{2\beta m}$, and for any $0<\varepsilon<1$, a BBM that starts with a single particle at the origin will escape this ball with a probability tending to one as $t\rightarrow\infty$ (see \cite{MK1975}).
\end{remark}

\begin{proof}
The strategy is to divide the time interval $[0,t]$ into two pieces: $[0,\delta t]$ and $[\delta t,t]$, and then to condition on the number of particles and the radius of the range at time $\delta t$. Here, $0<\delta<1$ is a number, which will later depend on $\varepsilon$. 

Let $A_t$ be the event that $Z$ avoids the trap field in the time interval $[0,t]$, and let $p(t):=P(A_t)$. For an upper bound on $p(t)$, we may suppose\footnote{By Brownian scaling, changing the distance of the trap is equivalent to speeding up or slowing down  time.} that $\bar{B}(0,\rho_t)$ contains precisely 1 point from $\text{supp}(\Pi)$, which is on the boundary of $\bar{B}(0,\rho_t)$. 

For $0<\delta'<\delta$ and $\delta''>0$, let $B_t$ be the event that at least $\lfloor e^{\beta \delta' t} \rfloor$ particles are produced in the time interval $[0,\delta t]$ and $C_t$ be the event that the BBM remains inside $B(0,(\delta+\delta'')t\sqrt{2\beta m})$ throughout $[0,\delta t]$. Use the estimate \[P(A)\leq P(A \mid B\cap C)+P(B^c)+P(C^c)\] to obtain
\begin{equation} 
p(t)\leq  P(A_t \mid \! B_t \cap C_t) + P(B_t^c) + P(C_t^c). \label{eq2}
\end{equation}

Let $p_3(t):=P(B_t^c)$ and $N(t):=|Z(t)|$. From \cite{KT1975}, for strictly dyadic branching (denote this process by $\tilde{N}$), we have
\begin{equation} P(\tilde{N}(t)> k)=(1-e^{-\beta t})^k \quad \text{for} \quad k=0,1,2,\ldots \label{eq3}
\end{equation}
Then, using the binomial theorem, we have 
\begin{equation} P\left(\tilde{N}\left(\delta t\right)\leq k\right)=1-\left(1-e^{-\beta \delta t}\right)^k\leq k e^{-\beta \delta t}  \quad \text{for} \quad k=0,1,2,\ldots \label{eq4}
\end{equation}
Setting $k=\lfloor e^{\delta'\beta t} \rfloor$, and comparing a BBM having $p_0=p_1=0$ (which holds by hypothesis) with a strictly dyadic BBM, we have for all $t>0$,
\begin{equation} p_3(t)\leq \exp[-(\delta-\delta')\beta t+o(t)]. \label{eq5}
\end{equation}

Let $p_4(t):=P(C_t^c)$. Define $M(t):=\inf\left\{r\geq 0:R(t)\subseteq B(0,r)\right\}$ to be the radius of the minimal ball containing the range of the BBM up to time $t$. Observe that
\begin{equation} p_4(t)=P\left(M(\delta t)>\sqrt{2\beta m}\left(1+\frac{\delta''}{\delta}\right)\delta t\right). \label{eq51}
\end{equation}
We now find an upper bound for $p_4(t)$. Let $N_t$ denote the set of particles that are alive at $t$ and for $1\leq u \leq |N_t|$, $X_u(t)$ denote the position of particle $u$ at time $t$. Then, using the union bound, for $\gamma>0$,
\begin{align} P\left(M(t)>\gamma t\right)
& = P\left(\exists u\in N_t\ :\ \sup_{0\le s\le t}|X_u(s)|>\gamma t\right) \nonumber \\
& \le E[N(t)]\:\mathbf{P}_0\left(\sup_{0\le s\le t}|B(s)|>\gamma t\right), \label{eq52}
\end{align}
where $B=(B(t))_{t\geq 0}$ represents standard Brownian motion starting at the origin, with probability $\mathbf P_0$. It is a standard result that $E[N(t)]=\exp(\beta m t)$ (see for example \cite[Sect.8.11]{KT1975}). Moreover, we know from \cite[Lemma 5]{OCE2017} that $\mathbf{P}_0\left(\sup_{0\le s\le t}|B(s)|>\gamma t\right)=\exp[-\gamma^2 t/2+o(t)]$. Then, by choosing $\gamma=\sqrt{2\beta m}\left(1+\frac{\delta''}{\delta}\right)$ and replacing $t$ by $\delta t$ in (\ref{eq52}), it follows from (\ref{eq51}) and (\ref{eq52}) that
\begin{equation} p_4(t)\leq \exp\left[-\beta m \delta t\left(\frac{\delta''^2}{\delta^2}+2\frac{\delta''}{\delta}\right)+o(t)\right]. \label{eq7}
\end{equation}

Now let $p_2(t):=P(A_t \mid B_t \cap C_t)$. Note that conditioned on the event $B_t \cap C_t$, there are at least $\lfloor e^{\beta \delta' t} \rfloor$ particles within the ball $B(0,(\delta+\delta'')t\sqrt{2\beta m})$ at time $\delta t$, each of which is \textbf{at most} at a distance 
\[ \sqrt{2\beta m}t(1-\varepsilon+\delta+\delta'')=:r(t) \] 
away from the trap point. (Recall that $\rho_t=\sqrt{2\beta m}(1-\varepsilon)t$.) Focus on one such particle. The probability that the sub-BBM emerging from this particle avoids the trap in the remaining time $t(1-\delta)$ is at most the sum of the probability that it remains in its $r(t)$-ball (call this $p_5(t)$) and the probability that it avoids the trap given that it escapes its $r(t)$-ball (call this $p_6(t)$). By the $r(t)$-ball, we mean the ball with radius $r(t)$ that is centered at the position of the particle at time $\delta t$. Hence, by the Markov property and independence of particles, we have
\begin{equation} p_2(t)\leq \left[p_5(t)+p_6(t)\right]^{\lfloor e^{\delta'\beta t} \rfloor}. \label{eq8}
\end{equation}

Consider $p_5(t)$. \textit{Choose $\delta>0$ and $\delta''>0$} so that $1-\varepsilon+\delta+\delta''<1-\delta$, which is equivalent to
\begin{equation}   2\delta+\delta''<\varepsilon. \label{eq81}
\end{equation}
Then, \cite[Prop.5]{E2004} implies\footnote{In \cite{E2004} the branching was strictly dyadic but the proof can be adapted easily to our more general setting.} that there exists a constant $c=c(\varepsilon,\beta,m)>0$ such that  
\begin{equation} p_5(t)\leq e^{-ct} \quad \text{for all large $t$.} \label{eq9}
\end{equation}
Now consider $p_6(t)$. By spherical symmetry, using the standard argument of proportion of surface areas, we have 
\begin{equation} p_6(t)\leq\left(1-\frac{\gamma_{a,d}}{r(t)^{d-1}}\right) \quad \text{for all $t>0$}, \label{eq10}
\end{equation} 
where $\gamma_{a,d}$ is a constant that depends on the dimension $d$ and the trap radius $a$.

From (\ref{eq8})-(\ref{eq10}), it is clear that there exists a constant $c=c(\varepsilon,a,d,\beta,m)>0$ such that for all large $t$, we have
\begin{equation} p_2(t)\leq \left(1-\frac{\gamma_{a,d}/2}{r(t)^{d-1}}\right)^{\lfloor e^{\delta'\beta t} \rfloor}\leq [\exp(-ct)]^{\lfloor e^{\delta'\beta t} \rfloor /t}=\exp(-c\lfloor e^{\delta'\beta t} \rfloor ) =\mathrm{SES}, \label{eq11}
\end{equation} 
where the estimate $(1-x/s)^s\leq e^{-x}$ is used in the second inequality and `SES' means `superexponentially small' in $t$. Also, note that the factor $1/2$ in the numerator in the second expression makes up for $p_5(t)$.

Now, putting everything together, from (\ref{eq2}), (\ref{eq5}), (\ref{eq7}) and (\ref{eq11}), we have  
\begin{equation} p(t)\leq \exp[-(\delta-\delta')\beta t+o(t)]+\exp\left[-\beta m \delta t\left(\frac{\delta''^2}{\delta^2}+2\frac{\delta''}{\delta}\right)+o(t)\right]+\text{SES}, \nonumber
\end{equation}
subject to the constraint $2\delta+\delta''<\varepsilon$, where the SES term comes from (\ref{eq11}). First, let $\delta'\rightarrow 0$ to obtain
\begin{equation} \underset{t\rightarrow\infty}{\limsup}\,\frac{1}{t}\log p(t)\leq -\beta\delta\min\left\{1,m\left(\frac{\delta''^2}{\delta^2}+2\frac{\delta''}{\delta}\right)\right\}. \label{eq121}
\end{equation}
Next, find the sharpest bound on $p(t)$ by optimizing over the parameters $\delta$ and $\delta''$, respecting the condition (\ref{eq81}). 


It is clear from (\ref{eq121}) that we need to maximize 
\begin{equation} f(\delta,\delta''):=\min\left\{\delta,m \delta \left(\frac{\delta''^2}{\delta^2}+2\frac{\delta''}{\delta}\right)\right\} \quad \text{subject to} \quad 2\delta+\delta''<\varepsilon. \nonumber 
\end{equation} 
Let $\delta''=\delta/k$, $k>0$ so that $f(\delta,\delta'')=f(\delta,k)=\min\left\{\delta,m \delta \left(\frac{1}{k^2}+\frac{2}{k}\right)\right\}$, and the constraint becomes $\delta<\varepsilon k/(2k+1)$. In order to maximize $f$, we solve 
\begin{equation}
1=m\left(\frac{1}{k^2}+\frac{2}{k}\right) \nonumber
\end{equation}
for positive $k$. This gives $k=m+\sqrt{m^2+m}$ as the optimal value for $k$, and the constraint becomes $\delta<\varepsilon(\sqrt{m^2+m}-m)$. By letting $\delta\rightarrow\varepsilon(\sqrt{m^2+m}-m)$, it follows from (\ref{eq121}) that
\begin{equation} \underset{t\rightarrow\infty}{\limsup}\,\frac{1}{t}\log p(t)\leq -\beta \varepsilon(\sqrt{m^2+m}-m). \nonumber
\end{equation}   
Indeed, by choosing $k$ differently, one obtains a weaker bound for $p(t)$. If $k>m+\sqrt{m^2+m}$, then $1>m\left(\frac{1}{k^2}+\frac{2}{k}\right)$ so that $f(\delta,k)=m \delta \left(\frac{1}{k^2}+\frac{2}{k}\right)$. In view of $\delta<\varepsilon k/(2k+1)$, we then have $p(t)\leq \exp[-\beta t \varepsilon \frac{m}{k}+o(t)]$, where $m/k<m/(m+\sqrt{m^2+m})$. Similarly, if $k<m+\sqrt{m^2+m}$, then $1<m\left(\frac{1}{k^2}+\frac{2}{k}\right)$ so that $f(\delta,k)=\delta$, and we have $p(t)\leq \exp[-\beta t \varepsilon \frac{k}{2k+1}+o(t)]$, where $k/(2k+1)<m/(m+\sqrt{m^2+m})$.  
\end{proof}

\begin{remark} Intuitively, what we are using in the proof of Lemma~\ref{lemma1} is that there are exponentially many particles at the frontier of a BBM instead of just 1 particle. In our proof, this appears as the factor $\lfloor e^{\beta \delta' t} \rfloor$ in (\ref{eq11}). Even though the BBM on average has $e^{\beta m t}$ particles at time $t$, and the ones on the frontier (meaning the ones that have escaped out of $B(0,\sqrt{2\beta m}t(1-\varepsilon))$) are not ``too many", they are not ``too few" either, there are still exponentially many ($\lfloor e^{\beta \delta' t} \rfloor$) on the frontier.
\end{remark}

\begin{remark} Lemma~\ref{lemma1} enables us to easily conclude the following: Let $0<\varepsilon<1$ and $\rho_t:=\sqrt{2\beta m}(1-\varepsilon)t$. Let $\Pi$ denote any Poisson random measure on $\mathbb{R}^d$ with mean measure $\nu$ such that the probability that $B(0,\rho_t)$ is trap-free is exponentially small in $t$. (For example, any $\nu$ that yields $\mathbb{P}(B(0,r)\:\:\text{is trap-free})\leq e^{-cr}$ for all $r>0$ for some $c>0$.) Then, for all large $t$, the annealed probability that the system avoids the trap field up to time $t$ is at most exponentially small in $t$, that is,
\[ (\mathbb{E}\times P)(\mathsf{S_t})\leq e^{-kt} \]
for some constant $k>0$ that possibly depends on $\varepsilon$, $\beta$, $m$ and $\nu$. Indeed, one easily obtains this result by conditioning on the event that $B(0,\rho_t)$ is trap-free, and applying Lemma~\ref{lemma1} on its complement. 
\end{remark}

\section{Proof of main theorems}\label{section4}

We now give the proof of the main theorems: Theorem~\ref{theorem1} and Theorem~\ref{theorem2}. The central ingredient in both proofs is Lemma~\ref{lemma1}. We give a bootstrap argument in each proof. Namely, in the proof of Theorem~\ref{theorem1}, we first show that for a given $\widehat\varepsilon>0$, with overwhelming probability, there is at most $k(\widehat\varepsilon)$ particles present at time $(1-\widehat\varepsilon)t$, where $k$ doesn't depend on time. Then, using this, we show that there is actually just one particle present. In the proof of Theorem~\ref{theorem2}, the same strategy is followed with the only difference being the replacement of $k(\widehat\varepsilon)$ by $\lfloor t \rfloor$. 

Next, we summarize the argument we will use, before turning to the proof of Theorem~\ref{theorem1}.

\vspace{5mm}

\noindent \textbf{A general `bootstrap argument'}

Let $\{A_t\}_{t>0}$ and $\{B_t\}_{t>0}$ be two families of events.
We are going to apply the  general argument that, since
$$\frac{P(B_t\mid \mathsf{S_t})}{P(A_t\mid \mathsf{S_t})}=\frac{P(B_t)}{P(A_t)}\frac{P( \mathsf{S_t}\mid B_t)}{P( \mathsf{S_t}\mid A_t)},$$
it follows that if $$\lim_{t\to\infty}\frac{P( \mathsf{S_t}\mid B_t)}{P( \mathsf{S_t}\mid A_t)}=0,$$ $\frac{P(B_t)}{P(A_t)}$ remains bounded from above and $\lim_{t\to\infty}P(A_t\mid \mathsf{S_t})=1$, then $\lim_{t\to\infty}P(B_t\mid \mathsf{S_t})=0$.
In other words, in this situation, if we know that $\{A_t\}_{t>0}$ is an optimal strategy for $\mathsf{S_t}$, then so is $\{A_t\cap B_t^c\}_{t>0}$.

This enables a `bootstrap' argument, namely, one first checks that $\{A_t\}_{t>0}$ is an optimal strategy, and then strengthens the argument by replacing $\{A_t\}_{t>0}$ with $\{A_t\cap B_t^c\}_{t>0}$.

\vspace{5mm}

\noindent \textbf{Proof of Theorem~\ref{theorem1}}

Fix $0<\widehat\varepsilon<1$.  Let $K_t:=\left\{|Z((1-\widehat\varepsilon)t)|\leq \lfloor (\sqrt{2}+1)/\widehat\varepsilon \rfloor \right\}$.
We first show that 
\begin{equation} \underset{t\rightarrow\infty}{\lim}(K_t \mid \mathsf{S_t})=1. \label{equ1}
\end{equation}
 By \eqref{001}, noting that $\alpha=1$ when $p_0=0$, it is enough to verify that 
\begin{equation}\label{eq14}
\underset{t\rightarrow\infty}{\limsup}\,\frac{1}{t}\log(\mathbb{E}\times P)(K_t^c\cap\mathsf{S_t})<-\beta. 
\end{equation}
Split the time interval $[0,t]$ into two pieces: $[0,(1-\widehat\varepsilon)t]$ and $[(1-\widehat\varepsilon)t,t]$. For $0<\varepsilon<1$, let $\widehat\rho_t=(1-\varepsilon)\sqrt{2\beta m}\widehat\varepsilon t$. Define $A_t$ to be the event that among the particles alive at time $(1-\widehat\varepsilon)t$, there is at least $1$ such that the ball $B(0,\widehat\rho_t)$ around it is trap-free. Estimate
\begin{equation} (\mathbb{E}\times P)(K_t^c\cap\mathsf{S_t})\leq (\mathbb{E}\times P)(A_t)+(\mathbb{E}\times P)(K_t^c\cap\mathsf{S_t}\mid A_t^c). \label{eq140}
\end{equation} 
By the definition of the Poisson random measure, and using the union bound, having a uniform intensity yields that for $t>0$, $(\mathbb{E}\times P)(A_t)\leq u(t)\exp[-ct^d]$, where 
$$u(t):=E\left|Z((1-\widehat\varepsilon)t)\right|=\exp\left[\beta m(1-\widehat\varepsilon)t\right],$$
and $c>0$ is some constant. Since $d\geq 2$ by assumption, it follows that the first term on the right-hand side of (\ref{eq140}) is SES. By Lemma~\ref{lemma1}, the second term on the right-hand side of (\ref{eq140}) is at most 
\begin{equation}
\exp\left[-\left(\lfloor (\sqrt{2}+1)/\widehat\varepsilon \rfloor+1\right)\beta\varepsilon(\sqrt{m^2+m}-m)\widehat\varepsilon t+o(t)\right]. \nonumber
\end{equation}
Since $m\geq 1$, observe that $\sqrt{2}-1\leq\sqrt{m^2+m}-m$. Finally, use (\ref{eq140}) and let $\varepsilon\rightarrow 1$ to obtain (\ref{eq14}), which completes the proof of (\ref{equ1}). (Note that the $\widehat\varepsilon$ appearing in the statement of the theorem is different from the $\varepsilon$ appearing in the definition of $\widehat\rho_t$.) 

\smallskip

Next, using (\ref{equ1}), we reduce the number of particles to $1$. Namely, we show that
\begin{equation} \underset{t\rightarrow\infty}{\lim}(\mathbb{E}\times P)\left(|Z((1-\widehat\varepsilon)t)|=1\mid \mathsf{S_t}\right)=1. \nonumber
\end{equation}

This is done by using the `bootstrap argument' (explained at the beginning of this section) with $A_t:=\left\{|Z((1-\widehat\varepsilon)t)|\le \lfloor (\sqrt{2}+1)/\widehat\varepsilon \rfloor\right\},\ B_t:=\left\{|Z((1-\widehat\varepsilon)t)|=k\right\}$, and with some $2\le k\le \lfloor (\sqrt{2}+1)/\widehat\varepsilon \rfloor$ fixed. We first note that 
\begin{equation}\label{same.cost.roughly}
0<\liminf_{t\to\infty} \frac{P(B_t)}{P(A_t)}\le \limsup_{t\to\infty} \frac{P(B_t)}{P(A_t)}<\infty,\ k\ge 1.
\end{equation}
This is clearly true for a Yule process (corresponding to $p_2=1$), as we have the explicit formula for the distribution of $|Z(t)|$ as $P(|Z(t)|=k)=e^{-\beta t}[1-e^{-\beta t}]^{k-1}$ for $t>0$ and $k\geq 1$. For a general supercritical process with $p_0=p_1=0$, \eqref{same.cost.roughly} follows by comparison with a Yule process as in the proof of \cite[Lemma 6]{OCE2017}, where the term $\lfloor t^{d+\varepsilon}  \rfloor$ therein should be replaced by $k$. 

On the other hand, surviving with even $2$ particles has an annealed probability which is lower order than surviving with one particle.
To see why this is true, suppose for simplicity that the remaining time is $t$ instead of $\widehat\varepsilon t$ (the argument is the same for $\widehat\varepsilon t$). Let $Y^x_t=Y^x_t(\omega)$ be the probability of survival up to $t$ for a BBM that starts with a single particle at $x$. Here, $\omega$ represents a realization of the trap field. By the Cauchy-Schwarz inequality and spatial homogeneity,  
\begin{equation} \mathbb{E}\left[Y^x_t Y^y_t\right]\leq \sqrt{\mathbb{E}\left[(Y^x_t)^2\right]\mathbb{E}\left[(Y^y_t)^2\right]} = \mathbb{E}[(Y^0_t)^2]. \label{equ}
\end{equation}
Now let $p_t=Y^0_t$. Thus, to prove that the probability of surviving with $2$ particles is of lower order than the probability of surviving with just $1$, it is enough to show that $\mathbb{E}[\: p_t^2]=o(\mathbb{E}[\: p_t])$ as $t\rightarrow \infty$. To this end, let $\varepsilon>0$ and $\rho_t:=(1-\varepsilon)\sqrt{2\beta m}\,t$ be a subcritical radius. Denote 
$$\Omega_t:=\{\omega\in\Omega\mid \text{supp}(\Pi(\omega))\cap \bar{B}(0,\rho_t)\neq \emptyset\}.$$
Then for all large $t$,
\begin{align}\label{with.indicators} \mathbb{E}[\: p_t^2]\leq & \:\mathbb{E}[\: p_t^2 \mathbbm{1}_{\Omega_t}]  + \mathbb{P}(\Omega_t^c) \leq  \: e^{-ct}  \:\mathbb{E}[\: p_t \mathbbm{1}_{\Omega_t}]+e^{-ct^d} \leq  \: e^{-ct}\mathbb{E}[\: p_t]+e^{-ct^d}, 
\end{align}
where all constants are denoted generically, as before, by $c$, and Lemma~\ref{lemma1} is used in passing to the second inequality: write $p_t^2=p_t p_t$ and on $\Omega_t$, bound the second $p_t$ from above by $e^{-ct}$, according to Lemma~\ref{lemma1}. Now, since $d\geq 2$, we have for all large $t$,
\begin{equation} \frac{\mathbb{E}[\: p_t^2]}{\mathbb{E}[\: p_t]}\leq e^{-ct}+\frac{e^{-ct^d}}{\mathbb{E}[\: p_t]}\rightarrow 0 \:\: \text{as} \:\: t\rightarrow \infty, \label{equ3}
\end{equation}
since we know from \eqref{001} that for a BBM with $p_0=p_1=0$ in $d\geq 2$, we have $\mathbb{E}[\: p_t]=\exp[-\beta t+o(t)]$. This shows that the probability of surviving with $2$ particles is of lower order in $t$ than that of surviving with just $1$.

\smallskip

To finish the proof, in view of (\ref{equ1}), we show that the probability of surviving with at most $\lfloor (\sqrt{2}+1)/\widehat\varepsilon \rfloor$ particles is of lower order than surviving with just $1$. It follows from (\ref{equ3}) that the convergence of the ratio $\frac{\mathbb{E}[\: p_t^2]}{\mathbb{E}[\: p_t]}$ to $0$ is at least exponentially fast in $t$. Hence, we conclude that for some constant $k>0$ and for all large $t$, we have
\begin{equation} \frac{\mathbb{E}[\: p_t^2]}{\mathbb{E}[\: p_t]}\leq e^{-kt}, \nonumber
\end{equation} 
and, {\it a fortiori},
\[ \frac{\mathbb{E}[\: p_t^j]}{\mathbb{E}[\: p_t]}\leq e^{-kt} \quad \text{for every} \quad 2\leq j\leq \lfloor (\sqrt{2}+1)/\widehat\varepsilon \rfloor \quad \text{for all large $t$},\]
which implies that 
\[\frac{\sum_{j=2}^{\lfloor (\sqrt{2}+1)/\widehat\varepsilon \rfloor} \mathbb{E}[\: p_t^j]}{\mathbb{E}[\: p_t]}\leq \lfloor (\sqrt{2}+1)/\widehat\varepsilon \rfloor e^{-kt}\rightarrow 0 \quad \text{as} \quad t\rightarrow\infty.  \]
This completes the bootstrap argument and shows that (\ref{equ1}) can be improved to 
\begin{equation} \underset{t\rightarrow\infty}{\lim}(\mathbb{E}\times P)\left(|Z((1-\widehat\varepsilon)t)|=1 \mid \mathsf{S_t}\right)=1. \eqno \qed  \nonumber  
\end{equation}

\vspace{5mm}

\noindent \textbf{Proof of Theorem~\ref{theorem2}}

Fix $0<\widehat\varepsilon<1$. Let $K_t=\left\{|Z((\eta^*-\widehat\varepsilon)t)|\leq f(t)\right\}$. First, we show that for any function $f:\mathbb{R}_+ \to \mathbb{R}_+$ such that $\underset{t\rightarrow\infty}{\lim}f(t)=\infty$, 
\begin{equation} \underset{t\rightarrow\infty}{\lim}(\mathbb{E}\times P)\left(K_t \mid \mathsf{S_t}\right)=1. \label{equ5}
\end{equation}
To show \eqref{equ5}, in view of \eqref{002}, it suffices to show that
\begin{equation} \underset{t\rightarrow\infty}{\limsup}\:\frac{1}{t}\log(\mathbb{E}\times P)\left(K_t^c \cap \mathsf{S_t} \right)<-I(l,f,\beta,d). \label{equ51}
\end{equation}
Obviously, we may (and will) assume that  $\underset{t\rightarrow\infty}{\lim}\frac{f(t)}{t^n}=0$ for some $n\in\mathbb{N}$.

We follow an argument similar to the one in Section 3.2 in \cite{E2003}. For $t\ge 0$, let $$\eta_t:=\sup\left\{\eta\in[0,1]:|Z(\eta t)|\leq f(t) \right\},$$ and notice that $\left\{\eta_t< x\right\}\subseteq\left\{|Z(xt)|> f(t)\right\}$ for $x\in(0,1]$, and that $K_t^c=\left\{\eta_t\leq\eta^*-\widehat\varepsilon\right\}$.
Introducing the conditional probabilities
\begin{equation} P^{(i,n)}_t(\cdot)=P\left(\ \cdot \ \middle| \ \frac{i}{n}\leq \eta_t <\frac{i+1}{n}\right), \quad i=0,1,\ldots,n-1, \nonumber
\end{equation} we have that, for every $n\in\left\{1,2,3,\ldots\right\}$,
\begin{align} (&\mathbb{E}\times P)(K_t^c \cap\mathsf{S_t}) \nonumber \\
\leq & \sum_{i=0}^{\left\lceil  (\eta^*-\widehat\varepsilon)n \right\rceil-1} (\mathbb{E}\times P)\left(\mathsf{S_t}\cap \left\{\frac{i}{n}\leq \eta_t <\frac{i+1}{n}\right\} \right)+(\mathbb{E}\times P)\left(\mathsf{S_t}\cap \left\{\eta_t=\eta^*-\widehat\varepsilon\right\} \right) \nonumber \\
\leq & \sum_{i=0}^{\left[ (\eta^*-\widehat\varepsilon)n \right\rceil-1} \exp\left[-\beta \frac{i}{n}t+o(t)\right](\mathbb{E}\times P^{(i,n)}_t)(\mathsf{S_t})+\exp\left[-\beta(\eta^*-\widehat\varepsilon)t+o(t)\right](\mathbb{E}\times P)(\mathsf{S_t}\mid \eta_t=\eta^*-\widehat\varepsilon ). \label{bigsum}
\end{align}
Consider the particles alive at time $t(i+1)/n$ (resp. $(\eta^*-\widehat\varepsilon)t$), and the balls with radius\footnote{I.e., the ball of critical radius for the remaining time.}
\begin{equation}
\rho_t^{(i,n)}:=(1-\varepsilon)\sqrt{2\beta m}\left(1-\frac{i+1}{n}\right)t, \quad
\text{resp.} \:\:\rho_t^*:=(1-\varepsilon)\sqrt{2\beta m}(1-\eta^*+\widehat\varepsilon)t \nonumber
\end{equation}
around them, and finally, let $\mathsf{TF}_t^{(i,n)}$ (resp. $\mathsf{TF}_t$) be the number of trap-free\footnote{In the sense that they do not receive points from $\Pi$.} balls among these. 
Define the events 
$$A_t^{(i,n)}:=\left\{\mathsf{TF}_t^{(i,n)}\ge 1\vee (|Z(t(i+1)/n)|-f(t))\right\};\ 
A_t:=\{\mathsf{TF}_t\ge 1\vee (|Z((\eta^*-\widehat\varepsilon)t)|-f(t))\}.$$
Use the trivial estimate
\begin{equation} \left(\mathbb{E}\times P^{(i,n)}_t\right)\left(\mathsf{S_t}\right)\leq \left(\mathbb{E}\times P^{(i,n)}_t\right)\left( A_t^{(i,n)}\right)+\left(\mathbb{E}\times P^{(i,n)}_t\right)\left(\mathsf{S_t} \mid [A_t^{(i,n)}]^c\right), \label{estimate}
\end{equation}
and a similar estimate for $(\mathbb{E}\times P)(\mathsf{S_t}\mid \eta_t=\eta^*-\widehat\varepsilon)$. Letting $\eta=i/n$, it is not hard to show that (see the proof of \cite[Thm.1]{OCE2017} for details) 
\begin{align}
\exp\left[-\beta \frac{i}{n}t+o(t)\right]&\left(\mathbb{E}\times P^{(i,n)}_t\right)\left( A_t^{(i,n)}\right)\leq  \nonumber \\
&\:\:\:\exp\left[-\underset{\eta \in [0,(\eta^*-\widehat\varepsilon)],c \in [0,\sqrt{2\beta}]}{\text{min}} \left\{\beta\eta+\frac{c^2}{2\eta}+l g_d(\sqrt{2\beta m}(1-\eta),c)\right\}+o(t)\right] \label{neweq}
\end{align}
(and similarly for $\left(\mathbb{E}\times P_t\right)\left( A_t\right)$). We know from \cite[Thm.2]{OCE2017} that $(\eta^*,c^*)$ is the unique pair of minimizers for the variational problem in \eqref{I}, and the parameter $\eta$ on the right-hand side of \eqref{neweq} is bounded away from $\eta^*$. Therefore, putting \eqref{estimate} and \eqref{neweq}  together with \eqref{bigsum}, to obtain \eqref{equ51}, it suffices to show that  
\begin{equation} \left(\mathbb{E}\times P^{(i,n)}_t\right)
\left(\mathsf{S_t} \mid \left[A_t^{i,n}\right]^c\right)=\text{SES} \quad \text{and} \quad 
\left(\mathbb{E}\times P\right)
\left(\mathsf{S_t} \mid [A_t]^c\right)=\text{SES} \label{eq0}
\end{equation}
 in $t$ for $i=0,1,2,\ldots,\left\lceil(\eta^*-\widehat\varepsilon)n \right\rceil-1$ for some large enough $n$. 

We now verify the first statement in (\ref{eq0}); the second could be verified similarly. Let $p^{i,n}(t)$ be the probability that a BBM, which starts its life at time $\frac{i+1}{n}t$ with a single particle at a point $x\in\mathbb{R}^d$, and whose $\rho_t$-ball (centered at $x$) receives a point from $\text{supp}(\Pi)$, avoids the trap field in the time interval $[\frac{i+1}{n}t,t]$. It is enough to show that $[p^{i,n}(t)]^{\lfloor f(t) \rfloor}$ is SES for $i=0,1,2,\ldots,\left\lceil(\eta^*-\widehat\varepsilon)n \right\rceil-1$. We may drop the floor function and work with $f(t)$ directly. Furthermore, we may work with the entire interval $[0,t]$ instead of $[\frac{i+1}{n}t,t]$. (It is enough to consider $[0,t]$ instead of the smaller interval $[\frac{i+1}{n}t,t]$ as this will not affect the final probabilistic cost being SES. In more detail, we show that for all large $t$, $p^{-1,n}(t)$ is bounded from above by $e^{-\kappa t}$ for some $\kappa>0$. If we consider the smaller interval $[\frac{i+1}{n}t,t]$, then $p^{i,n}(t)$ will be bounded by $e^{-\kappa't}$, where $0<\kappa'<\kappa$.) Now let $p(t):=p^{-1,n}(t)$ and $\rho_t:=\rho_t^{(-1,n)}$. Note that since we are conditioning only on the event that the $\rho_t$-ball around the particle contains a point from $\text{supp}(\Pi)$, we may suppose that $x=0$, that is, our problem becomes the trap-avoiding probability of a BBM, starting with a single particle at the origin, presuming that $\bar{B}(0,\rho_t)\cap \text{supp}(\Pi)\neq \emptyset$. Now, by Lemma 1, $p(t)$ is at most exponentially small in $t$,  and since by assumption, $\underset{t\rightarrow\infty}{\lim}f(t)=\infty$, it follows that
\begin{equation} (p(t))^{f(t)} \quad \text{is} \quad \text{SES}.\nonumber
\end{equation}
This completes the proof of (\ref{equ5}).

\smallskip

Next, following a similar strategy as in the proof of Theorem~\ref{theorem1}, we reduce the number of particles to $1$, i.e., we show that
\begin{equation} \underset{t\rightarrow\infty}{\lim}(\mathbb{E}\times P)\left(|Z((1-\widehat\varepsilon)t)|=1\mid \mathsf{S_t}\right)=1. \nonumber
\end{equation}
We consider the cases $d=1$ and $d\geq 2$ separately, since for $d\geq 2$, the trap field is not uniform due to (\ref{eq00}), whereas for $d=1$, it is. In what follows, we use the notation from the proof of Theorem~\ref{theorem1}. 

\bigskip

\noindent \underline{The case $d=1$ and $l>l_{cr}$}. 

\medskip
If $d=1$, then (\ref{equ}) holds. In view of this, we first show that $\mathbb{E}[\: p_t^2]=o(\mathbb{E}[\: p_t])$ as $t\rightarrow \infty$. Let $\varepsilon>0$ and $\rho_t:=(1-\varepsilon)\sqrt{2\beta m}t$ be a subcritical radius. 
Then, the same calculation as in \eqref{with.indicators} yields that for all large $t$,
$$\mathbb{E}[\: p_t^2]\leq  \: e^{-ct}\mathbb{E}[\: p_t]+\exp\left[-2 l (1-\varepsilon)\sqrt{2\beta m}t\right],$$
where $l$ is the constant in  the trap intensity. Now, since $l>l_{cr}$, we put $l-l_{cr}=:\delta>0$. From \cite{OCE2017}, we know that, when $d=1$, the variational problem in \eqref{I} exhibits a crossover at $l_{cr}=\frac{1}{2}\sqrt{\beta/(2m)}$. Therefore, choose $\varepsilon$ small enough ($0<\varepsilon<\min\left\{1/2,\delta\sqrt{2\beta m}\right\}$ will suffice) so that
\begin{align} 2 l (1-\varepsilon)\sqrt{2\beta m}&=2 (\delta+l_{cr}) (1-\varepsilon)\sqrt{2\beta m} \nonumber \\
&=2 \delta (1-\varepsilon)\sqrt{2\beta m}+\beta (1-\varepsilon)>\beta. \nonumber
\end{align}
When $d=1$ and $l>l_{cr}$, for a BBM with $p_0=p_1=0$, we know from \cite[Thm.2.2]{OCE2017} that \eqref{I} becomes $I=\beta$, meaning that $\mathbb{E}[\: p_t]=\exp[-\beta t+o(t)]$. Therefore, 
\begin{equation} \frac{\mathbb{E}[\: p_t^2]}{\mathbb{E}[\: p_t]}\leq e^{-ct}+\frac{\exp[2 l (1-\varepsilon)\sqrt{2\beta m}t]}{\mathbb{E}[\: p_t]}\rightarrow 0 \:\: \text{as} \:\: t\rightarrow \infty, \nonumber
\end{equation}
where the convergence of the ratio $\frac{\mathbb{E}[\: p_t^2]}{\mathbb{E}[\:p_t]}$ to $0$ is at least exponentially fast in $t$. Hence, we conclude that for all large $t$, we have
\begin{equation} \frac{\mathbb{E}[\:p_t^2]}{\mathbb{E}[\:p_t]}\leq e^{-kt} \label{equ7}
\end{equation} 
for some $k>0$. Now let $f(t)=\lfloor t \rfloor$. Then, (\ref{equ5}) gives:
\begin{equation} \underset{t\rightarrow\infty}{\lim}(\mathbb{E}\times P)\left(|Z((1-\widehat\varepsilon)t)|\leq \lfloor t \rfloor \mid \mathsf{S_t}\right)=1. \label{equ8}
\end{equation}
It then follows from (\ref{equ7}) that  
\[ \frac{\mathbb{E}[\: p_t^j]}{\mathbb{E}[\: p_t]}\leq e^{-kt} \quad \text{for every} \quad 2\leq j\leq \lfloor t \rfloor \quad \text{for all large $t$,}\]
which implies that
\begin{equation}\frac{\sum_{j=2}^{\lfloor t \rfloor} \mathbb{E}[\: p_t^j]}{\mathbb{E}[\: p_t]}\leq \lfloor t \rfloor e^{-kt} \quad \text{for all large $t$}.   \label{equ81}
\end{equation}
Next, using the proof of \cite[Lemma 6]{OCE2017}, with the replacement of the term $\lfloor t^{d+\varepsilon} \rfloor$ therein by $\lfloor t \rfloor$, we see that 
\begin{equation} \frac{P(|Z(t)|\leq \lfloor t \rfloor )}{P(|Z(t)|=1)}\leq \lfloor t \rfloor \quad \text{for all large $t$} \label{equ82}. 
\end{equation}  
Finally, writing
\begin{align} & \frac{(\mathbb{E}\times P)\left(2\leq |Z((1-\widehat\varepsilon)t)|\leq \lfloor t \rfloor \mid \mathsf{S_t}\right)}{(\mathbb{E}\times P)\left(|Z((1-\widehat\varepsilon)t)=1 \mid \mathsf{S_t}\right)} = \nonumber \\
&\quad \quad \quad \frac{(\mathbb{E}\times P)\left(\mathsf{S_t} \mid 2\leq |Z((1-\widehat\varepsilon)t)|\leq \lfloor t \rfloor \right)}{(\mathbb{E}\times P)\left(\mathsf{S_t} \mid |Z((1-\widehat\varepsilon)t)|=1 \right)}\cdot\frac{P(2\leq |Z((1-\widehat\varepsilon)t)|\leq \lfloor t \rfloor)}{P(|Z((1-\varepsilon)t)|=1)}, \nonumber
\end{align}
it follows from (\ref{equ81}) and (\ref{equ82}) that (\ref{equ8}) can be improved as 
\begin{equation} \underset{t\rightarrow\infty}{\lim}(\mathbb{E}\times P)\left(|Z((1-\widehat\varepsilon)t)|=1 \mid \mathsf{S_t}\right)=1.  \nonumber 
\end{equation}

\bigskip

\noindent \underline{The case $d\geq 2$ and $l>l_{cr}$}. 

\medskip
When $d\geq 2$, the trap intensity is no longer uniform; instead, it is radially decaying. As before, let $Y_s^x=Y_s^x(\omega)$ be the probability of survival up to time $s\geq 0$ for a BBM that starts with a single particle at position $x\in\mathbb{R}^d$. Like before,  Cauchy-Schwarz yields
\begin{equation} \mathbb{E}[Y^x_t Y^y_t]\leq \sqrt{\mathbb{E}[(Y^x_t)^2]\mathbb{E}[(Y^y_t)^2]}, \nonumber
\end{equation}
however, unlike in the case of the uniform trap field, we may not replace $Y^x_t$ by $Y^0_t$. Instead, we proceed as follows. For fixed $0<\widehat\varepsilon<\eta^*$, let 
\[Y^{X(s)}:=Y_{(1-\eta^*+\widehat\varepsilon)t}^{X(s)}, \]
where $X=(X(s))_{s\geq 0}$ represents a standard Brownian path starting at the origin. Let $\widehat{X}_t:=X((\eta^*-\widehat\varepsilon)t)$. We want to show that there exists a constant $c>0$ such that
\begin{equation} \frac{\left(\mathbb{E}\times E\right)\left [Y^{\widehat{X^1_t}} Y^{\widehat{X^2_t}}\right]}{\left(\mathbb{E}\times E\right)\left[Y^{\widehat{X}_t}\right]}\leq e^{-ct} \quad \text{for all large $t$}, \nonumber
\end{equation} 
where $\widehat{X^1_t}$ and $\widehat{X^2_t}$ represent the positions of the particles present at time $\left(\eta^*-\widehat\varepsilon\right)t$, and the expectation $E$ is placed for the purpose of averaging over these positions. We note that $\widehat{X^1_t}$ and $\widehat{X^2_t}$ are dependent random variables, yet they are both identically distributed as $\widehat{X}_t$. Using the  inequality between the geometric and quadratic means ($ab\leq (a^2+b^2)/2$), we have
\begin{equation} (\mathbb{E}\times E) [Y^{\widehat{X^1_t}} Y^{\widehat{X^2_t}}]\leq \frac{1}{2}(\mathbb{E}\times E) \left[\left(Y^{\widehat{X^1_t}}\right)^2 + \left(Y^{\widehat{X^2_t}}\right)^2\right]=(\mathbb{E}\times E)\left[\left(Y^{\widehat{X}_t}\right)^2\right]. \nonumber
\end{equation} 
Hence, letting $Y:=Y^{\widehat{X}_t}$, it suffices to show that
\begin{equation} \frac{(\mathbb{E}\times E) [Y^2]}{(\mathbb{E}\times E)[Y]}\leq e^{-ct} \quad \text{for all large $t$} \quad \text{for some $c>0$}. \label{1}
\end{equation}   
Let $\rho_t^*:=\sqrt{2\beta m}(1-\varepsilon)(1-\eta^*+\widehat\varepsilon)t$ and denote the event 
$$\widehat\Omega_t:=\{\text{supp}(\Pi)\cap \bar{B}(\widehat{X}_t,\rho_t^*)\neq \emptyset\}.$$ 
To bound the numerator of (\ref{1}) from above, we write
\begin{align} 
(\mathbb{E}\times E)\left[Y^2\right]&=\:(\mathbb{E}\times E)\left[Y^2\mathbbm{1}_{\widehat\Omega_t}\right]+(\mathbb{E}\times E)\left[Y^2\mathbbm{1}_{\widehat\Omega_t^c}\right] \nonumber \\   
& \leq \: e^{-ct}(\mathbb{E}\times E)[Y]+ (\mathbb{E}\times E)\left[\mathbbm{1}_{\widehat\Omega_t^c}\right]  = \: e^{-ct}(\mathbb{E}\times E)[Y]+(E\times \mathbb{P})(\widehat\Omega_t^c)   \label{0}
\end{align}
for all large $t$, where we have used Lemma~\ref{lemma1}, which implies that with $(\mathbb{P}\times P)$-probability $1$, $Y\mathbbm{1}_{\widehat\Omega_t}\leq e^{-ct}$ for some $c>0$ for all large $t$, in passing to the first inequality, and Fubini's theorem in passing to the last equality. By conditioning $\widehat{X}_t$ on the events $\left\{\frac{i-1}{n}\sqrt{2\beta}t\leq \widehat{X}_t \leq \frac{i}{n}\sqrt{2\beta}t\right\}$ for $i=1,2,\ldots,n$ and on $\left\{\widehat{X}_t>\sqrt{2\beta}t\right\}$, and following an argument similar to the proof of the upper bound of Theorem 1 in \cite{OCE2017}, it is not hard to show that 
\begin{align} & (E\times \mathbb{P})(\widehat\Omega_t^c) \nonumber \\
= & \: \exp\left[-\min_{x\in[0,\sqrt{2\beta}]}\left\{\frac{x^2}{2(\eta^*-\widehat\varepsilon)}+l \, g_d\left(\rho_t^*/t,x \right) \right\} t +o(t) \right]+\exp[-\frac{\beta}{\eta^*-\widehat\varepsilon} t +o(t)], \label{2}
\end{align} 
where \cite[Lemma 5]{OCE2017} was used to control the probabilistic cost of linear Brownian displacements.

\smallskip

To bound the denominator of (\ref{1}) from below, we proceed as follows:
\begin{align} (\mathbb{E}\times E)\left[Y\right]\geq & \: P\left(\widehat{X}_t\geq \left(1-\frac{\widehat\varepsilon}{\eta^*}\right)c^* t\right)(\mathbb{E}\times E)\left[Y \mid \widehat{X}_t\geq \left(1-\frac{\widehat\varepsilon}{\eta^*}\right)c^* t\right]  \nonumber \\
= & \: \exp\left[-\frac{(1-\frac{\widehat\varepsilon}{\eta^*})^2 (c^*)^2}{2(\eta^*-\widehat\varepsilon)}t +o(t) \right] \nonumber \\
& \times \exp\left[-\left\{\beta \widehat\varepsilon + \frac{(\widehat\varepsilon c^*/\eta^*)^2}{2\widehat\varepsilon} + l \, g_d\left(\sqrt{2\beta m}(1-\eta^*),c^* \right)\right\}t +o(t) \right], \label{3}
 \end{align}
where we have used the following survival strategy to bound the second factor on the right-hand side: 
\begin{itemize}
\item suppress the branching of the BBM;
\item  move the single particle to an extra distance of 
$(\widehat\varepsilon c^*/\eta^*)t$ in the time interval $[(\eta^*-\widehat\varepsilon)t,\eta^*t]$, 
where the extra distance is in the same direction as the position vector of the single particle at time 
$(\eta^*-\widehat\varepsilon)t$; 
\item empty the region $B(c^*t \textbf{e},\sqrt{2\beta m}(1-\eta^*)t+\delta t)$ from traps, where $\delta>0$ and $\textbf{e}$ is the unit vector in the direction of the position vector of the single particle at time $(\eta^*-\widehat\varepsilon)t$;  
\item let the system branch freely in the remaining time interval $[\eta^*t,t]$ inside this ball. 
\end{itemize}
 (For details regarding this type of survival strategy, please see \cite[Sect.5.1]{OCE2017}.) 
Finally, let $\delta\rightarrow 0$.
\smallskip

Note that the distances $(1-\widehat\varepsilon/\eta^*)c^*=:x_1$ and $\widehat\varepsilon c^*/\eta^*=:x_2$ were chosen so as to satisfy the system
$$
x_1+x_2= \: c^*;\qquad
\frac{x_1^2}{2(\eta^*-\widehat\varepsilon)} +\frac{x_2^2}{2\widehat\varepsilon}= \: \frac{(c^*)^2}{2\eta^*},  
$$
so that (\ref{3}) becomes 
\begin{equation} (\mathbb{E}\times E)\left[Y\right]\geq \exp\left[-\left\{\beta \widehat\varepsilon + \frac{(c^*)^2}{2\eta^*} + l \, g_d\left(\sqrt{2\beta m}(1-\eta^*),c^* \right)\right\}t +o(t) \right]. \label{4}
\end{equation}
To prove that 
\begin{equation} (E\times \mathbb{P})\left(\widehat\Omega_t^c\right)=o\left((\mathbb{E}\times E)\left[Y\right]\right),  \label{5}
\end{equation}
we multiply the right-hand sides of (\ref{2}) and (\ref{4}) both by $\exp[-\beta(\eta^*-\widehat\varepsilon)t]$, and since this factor doesn't depend on the minimizing parameter in (\ref{2}) and since $\eta^*<1$ so that the second term on the right-hand side of (\ref{2}) is harmless, it is enough to show that 
\begin{align} & \min_{x\in[0,\sqrt{2\beta}]}\left\{\beta(\eta^*-\widehat\varepsilon) + \frac{x^2}{2(\eta^*-\widehat\varepsilon)} + l \, g_d\left(\sqrt{2\beta m}(1-\eta^*+\widehat\varepsilon),x \right) \right\} \nonumber \\
& \: \qquad \qquad > \beta \eta^* + \frac{(c^*)^2}{2\eta^*} + l \, g_d\left(\sqrt{2\beta m}(1-\eta^*),c^* \right). \label{6} 
\end{align} 
(Above, in writing the function $g_d$ from (\ref{2}), we have used that $\rho_t^*=\sqrt{2\beta m}(1-\varepsilon)(1-\eta^*+\widehat\varepsilon)t$ and then let $\varepsilon\rightarrow 0$.) Now (\ref{5}) follows, because we know from \cite[Thm.2]{OCE2017} that the pair $(\eta^*,c^*)$ is the unique pair of minimizers for the variational problem
\begin{equation}
\underset{\eta \in [0,1],c \in [0,\sqrt{2\beta}]}{\text{min}} \left\{\beta\eta+\frac{c^2}{2\eta}+l g_d(\sqrt{2\beta m}(1-\eta),c)\right\}. \label{7}
\end{equation}
However, on the left-hand side of (\ref{6}), we have the evaluation of the function in (\ref{7}) at $(\eta,c)=(\eta^*-\widehat\varepsilon,x)$ for some $x\in[0,\sqrt{2\beta}]$, and (\ref{5}) then follows by the uniqueness of minimizers. To obtain (\ref{1}), use (\ref{0}) and (\ref{5}). To complete the proof, follow a similar argument as in the last part of the proof of the case $d=1$ and $l>l_{cr}$. \qed

\begin{remark}
As we have noted in the proof of Theorem~\ref{theorem1}, for a `free' BBM, the probabilistic cost of having $1$ particle and at most $k$ particles are asymptotically similar up to a constant as $t\rightarrow\infty$. What Theorem~\ref{theorem1} and Theorem~\ref{theorem2} say is that, for the trap fields considered here, for large $t$, whenever the system has to suppress branching in order to survive from traps up to $t$, with overwhelming probability, it must do so completely up to time $(1-\widehat\varepsilon)t$ (resp. $(\eta^*-\widehat\varepsilon)t$). Furthermore, the proofs reveal that having even $2$ particles instead of $1$ at $(1-\widehat\varepsilon)t$ (resp. $(\eta^*-\widehat\varepsilon)t$) is exponentially unlikely in $t$. This shows that conditioning a BBM on survival among traps has a drastic effect on its population size. 
\end{remark}

\begin{remark} The proofs of Theorem~\ref{theorem1} and Theorem~\ref{theorem2} reveal something stronger than the statement of the theorems; namely, that conditional on survival up to time $t$, the probability of the respective complement events $\left\{|Z((1-\widehat\varepsilon)t)|>1\right\}$ and $\left\{|Z((\eta^*-\widehat\varepsilon)t)|>1\right\}$ converge to zero exponentially fast in $t$.
\end{remark}

\section{Particle production along a skeletal line}\label{section5}

Theorem~\ref{theorem1} and Theorem~\ref{theorem2} are stated for $p_0=0$. In Section~\ref{section6}, they will be extended to the case where $p_0>0$ (see Theorem~\ref{theorem3} and Theorem~\ref{theorem4}), which yields a positive probability of extinction for the BBM. In this case, we condition the BBM on non-extinction for meaningful results on optimal survival strategies. A detailed treatment of a BBM conditioned on non-extinction is given in \cite{OCE2017} (see in particular Lemma 4 and Proposition 2 therein). Here, in preparation for Section~\ref{section6}, we briefly mention the development needed, followed by the statement and proof of Lemma~\ref{lemma2}. Conditioned on the event of non-extinction (denoted by $\mathcal{E}^c$), recall that the BBM has the following two-type decomposition:
\begin{equation} (Z(t))_{t\geq 0}=\left(Z^1(t),Z^2(t)\right)_{t\geq 0}, \nonumber
\end{equation}
where $Z^1$ is the process consisting of the `skeleton' particles, and $Z^2$ is the one consisting of the `doomed' particles. Skeleton particles are those with infinite lines of descent, whereas the doomed particles have finite lines of descent. We refer to the totality of all skeleton particles as the `skeleton' so that the tree of $|Z|$ conditioned on non-extinction can be described as an infinite skeleton decorated with infinitely many finite `bushes' composed of doomed particles. It is clear that conditioning $Z$ on $\mathcal{E}^c$ is equivalent to the initial condition $\left(|Z^1(0)|,|Z^2(0)|\right)=(1,0)$. By a `skeletal ancestral line up to time $t$', we mean the continuous trajectory traversed up to time $t$ by a skeleton particle present at time $t$, concatenated with the trajectories of all its ancestors including the one traversed by the initial particle. We use the term `skeletal line' in short to mean a skeletal ancestral line up to time $t$. We say that a doomed particle is \textit{produced} by a skeletal line if the most recent skeleton ancestor of the doomed particle is a part of this skeletal line. Note that by this definition, a doomed particle may be produced by more than one skeletal line, but it has to be produced by at least one skeletal line. The following lemma gives an upper bound on the number of doomed particles, all alive at the present time, which are produced by a given single skeletal line.

\begin{lemma}[Very few doomed particles] \label{lemma2}Let $\log^{(0)}(t):=t$ and $\log^{(n)}(t):=\log(\log(\ldots(\log t)\ldots))$ for $n\in\mathbb{N}$ be the logarithm function iterated $n$ times. Then, for any $n\in\mathbb{N}_0$, for a fixed skeletal line, the probability that this line has produced more than $\log^{(n)}(t)$ doomed particles in $[0,t]$, which are all alive at time $t$, goes to zero at least at the rate $1/\log^{(n)}(t)$ as $t\rightarrow\infty$.  
\end{lemma}

In order to prove Lemma~\ref{lemma2}, we first present two preparatory propositions. 
The first provides an upper bound on the non-extinction probability of a subcritical BBM up to time $t$, and follows from the trivial estimate $P(|Z(t)|>0)\leq E[\,|Z(t)|\,]$; the second follows directly from a standard Poissonian tail bound.

\begin{proposition}\label{proposition1} Let $Z$ be a subcritical BBM with rate $\beta>0$ and offspring p.g.f.\ $f$, and $|Z|$ be the associated total-mass process. Let $\mu=f'(1)$ be the mean number of offspring so that $m:=\mu-1<0$. Then, for any $t\geq 0$,
\begin{equation} P(|Z(t)|>0)\leq e^{\beta m t}. \nonumber
\end{equation}
\end{proposition}


\begin{remark} For precise results on $P(|Z(t)|>0)$, please see \cite[Thm.2.4]{AH1983}.
\end{remark}


\begin{proposition}[Tail estimate]\label{proposition2} Let  $Y$ be  a Poisson random variable with parameter $\lambda$. Then for $x>\lambda$,
\begin{equation} P(Y\geq x)\leq e^{-k\lambda}, \nonumber
\end{equation}
where $k=k(x/\lambda)$ is a positive constant.
\end{proposition}

\begin{proof} Let $z:=\frac{\lambda}{x}\in (0,1)$. By the standard Poissonian tail estimate,
\begin{equation} \log P(Y\geq x)\leq \log\left[\frac{e^{-\lambda}(e\lambda)^x}{x^x}\right]=\lambda \left(-1+x\,\frac{1+\log \lambda}{\lambda}-x\log x/\lambda\right)=-\lambda k(z), \nonumber
\end{equation}
where $$k(z):=1-z^{-1}(1+\log z)>1-z^{-1}z=0,$$
as  $z>1+\log z$.
\end{proof}

\vspace{5mm}

\noindent \textbf{Proof of Lemma~\ref{lemma2}}. 

\noindent We prove the statement by an inductive argument as follows. Fix a single skeletal ancestral line. In this proof, by a doomed particle born `directly' along this skeletal line, we refer to a doomed particle whose direct ancestor is a skeleton particle of this line and by a `doomed subtree', we refer to a subtree that is initiated by a doomed particle born directly along this fixed skeletal line. Let $E_1$ be the event that the doomed subtrees created in the interval $I_1:=[0,t+\frac{4}{\beta m}\log t]$ do not all go extinct by time $t$ (recall that $m<0$) and let $P_1:=P(E_1)$. Let $F_1$ be the event that $\leq 2\beta t$ occurrences of branching occur along the skeletal line in the time interval $I_1$. Estimate 
\begin{equation} P_1\leq P(F_1^c)+P(E_1\mid F_1) \label{eq16}.
\end{equation}
By Proposition~\ref{proposition2}, since the number of occurrences of branching up to time $t$ along a single skeletal line is a Poisson process with mean $\beta t$, we have 
\begin{equation}
P(F_1^c)\leq e^{-k_1(1) t}, \label{eq161}
\end{equation}
where $k_1(1)>0$ is a constant that depends on $\beta$. Now focus on $P(E_1\mid F_1)$. Let $G_1$ be the event that $\leq t^3$ doomed subtrees are born in the interval $I_1$. Estimate
\begin{equation} P(E_1\mid F_1)\leq P(G_1^c\mid F_1)+P(E_1\mid G_1,F_1) \label{eq17}.
\end{equation} 
Let $\rho$ be the expected number of doomed offspring for a skeleton particle. (From \cite{OCE2017}, we know that $\rho=[f'(1)-f'(q)]q/(1-q)$, where $q$ is the probability of extinction for $Z$.) The first term on the right-hand side of (\ref{eq17}) is bounded from above by the probability that at least one skeletal branching among $2\beta t$ many gives at least $t^3/(2\beta t)=t^2/(2\beta)$ doomed offspring, which, by the union bound and Markov inequality, is bounded from above to yield
\begin{equation} P(G_1^c\mid F_1)\leq 2\beta t \frac{\rho}{t^2/(2\beta)}=k_2(1)/t, \label{eq171}
\end{equation}
where $k_2(1)$ is a constant that depends on $\beta$ and $f$. The second term on the right-hand side of (\ref{eq17}) is bounded from above by the probability that the doomed subtrees created in the interval $I_1$, of which there are at most $t^3$ many, do not all go extinct by $t$, which, by the union bound and Proposition~\ref{proposition1} (note that each doomed subtree is a subcritial BBM), is bounded from above to yield 
\begin{equation}
P(E_1\mid G_1,F_1)\leq t^3 \exp\left(\beta m \frac{-4}{\beta m}\log t\right)=1/t.  \label{eq172}
\end{equation}
Putting the pieces together, from (\ref{eq16})-(\ref{eq172}), we obtain
\begin{equation} P_1\leq e^{-k_1(1) t}+k_2(1)/t+1/t,\label{eq173}
\end{equation}
which implies that the doomed subtrees created in $I_1=[0,t+\frac{4}{\beta m}\log t]$ all go extinct by time $t$ with a probability tending to 1 as $t\rightarrow \infty$.

We now extend the argument above to the doomed subtrees created in the interval 
$I_n:=\left[t+\frac{4}{\beta m}\log^{(n-1)}t,t+\frac{4}{\beta m}\log^{(n)} t\right]$ for $n\geq 2$. For $n\geq 2$, let $E_n$ be the event that the doomed subtrees created in the interval $I_n$ do not all go extinct by time $t$ and let $P_n:=P(E_n)$. Let $F_n$ be the event that $\leq 2\beta \frac{-4}{\beta m}\log^{(n-1)}t=(-8/m)\log^{(n-1)}t $ occurrences of branching occur along the skeletal line in the time interval $I_n$. Estimate 
\begin{equation} P_n\leq P(F_n^c)+P(E_n\mid F_n) \label{eq18}.
\end{equation}  
By Proposition~\ref{proposition2}, since the number of occurrences of branching in $I_n$ along a single skeletal line is a Poisson process with mean $\leq (-4/m)\log^{(n-1)}t$, we have 
\begin{align} P(F_n^c) \leq
\begin{cases} & 1/t^{k_1(2)} \:,\: n=2,  \\
& 1/(\log^{(n-2)}t)^{k_1(n)} \:,\: n\geq 3,
\end{cases} \label{eq181}
\end{align}
where $k_1(n)>0$ is a constant that depends on $\beta$. Now focus on $P(E_n\mid F_n)$. Let $G_n$ be the event that $\leq (\log^{(n-1)}t)^3$ doomed subtrees are born in the interval $I_n$. Estimate
\begin{equation} P(E_n\mid F_n)\leq P(G_n^c\mid F_n)+P(E_n\mid G_n \cap F_n) \label{eq19}.
\end{equation} 
The first term on the right-hand side of (\ref{eq19}) is bounded from above by the probability that at least one skeletal branching among $(-8/m)\log^{(n-1)}t$ many gives at least \newline
$(\log^{(n-1)}t)^3/((-8/m)\log^{(n-1)}t)=(-m/8)(\log^{(n-1)}t)^2$ doomed offspring, which, by the union bound and Markov inequality, is bounded from above to yield
\begin{equation} P(G_n^c\mid F_n)\leq (-8/m)\log^{(n-1)}t\, \frac{\rho}{(-m/8)(\log^{(n-1)}t)^2}=k_2(n)/(\log^{(n-1)}t), \label{eq191}
\end{equation}
where $k_2(n)$ is a constant that depends on $f$. The second term on the right-hand side of (\ref{eq19}) is bounded from above by the probability that the doomed subtrees created in the interval $I_n$, of which there are at most $(\log^{(n-1)}t)^3$ many, do not all go extinct by $t$, which, by the union bound and Proposition~\ref{proposition1}, is bounded from above to yield
\begin{equation} P(E_n\mid G_n\cap F_n)\leq(\log^{(n-1)}t)^3 \exp\left(\beta m \frac{-4}{\beta m}\log^{(n)} t\right)=1/(\log^{(n-1)}t). \label{eq192}
\end{equation}
Then, from (\ref{eq18})-(\ref{eq192}), we obtain
\begin{align} P_n \leq
\begin{cases} & 1/t^{k_1(2)}+k_2(2)/\log t+1/\log t \: ,\: n=2,  \\
& 1/(\log^{(n-2)}t)^{k_1(n)}+k_2(n)/\log^{(n-1)}t+1/\log^{(n-1)}t \: ,\: n\geq 3.
\end{cases} \label{eq20}
\end{align}
This implies that the doomed subtrees produced by the skeletal line in $I_n$ have all gone extinct by time $t$ with a probability tending to 1 as $t\rightarrow \infty$. We recall that $I_1:=[0,t+\frac{4}{\beta m}\log t]$ and $I_n=[t+\frac{4}{\beta m}\log^{(n-1)}t,t+\frac{4}{\beta m}\log^{(n)} t]$ for $n\geq 2$ to conclude the following: for any $n\geq 1$, as $t\rightarrow\infty$,
\begin{equation} P(\text{doomed subtrees born in $[0,t-\log^{(n)} t]$ have all gone extinct by time $t$})\rightarrow 1. \label{eq21}
\end{equation}
The convergence in (\ref{eq21}) can easily be seen from (\ref{eq173}) and (\ref{eq20}) to be at least at the rate $1/t$ for $n=1$, and $1/(\log^{(n-1)} t)$ for $n\geq 2$. In view of (\ref{eq21}), since each doomed particle that is produced by the skeletal line is a member of a doomed subtree, each doomed particle present at time $t$ is a member of a doomed subtree that is created in the interval $[t-\log^{(n+1)} t,t]$ with probability tending to $1$ as $t\rightarrow\infty$. The result follows by applying similar bounds as above on the total progeny generated by the doomed subtrees produced along the skeletal line in the interval $[t-\log^{(n+1)} t,t]$; one just needs to multiply $\rho$ by the expected total progeny of a doomed subtree, which is finite as well. \qed

\section{Extension to the case $p_0>0$}\label{section6}

In this section, Theorem~\ref{theorem1} and Theorem~\ref{theorem2} are extended to the case $p_0>0$, where the probability of extinction for the BBM is positive. We condition the BBM on non-extinction $\mathcal{E}^c$ for meaningful results on optimal survival strategies. Recall that $Z$ has the offspring p.g.f.\ $f$, where $f(s)=\sum_{j=0}^\infty p_k s^k$ for $s\in[0,1]$. Suppose that $p_0>0$ and $\mu=f'(1)>1$. Let $Z=(Z^1,Z^2)$ be the decomposition of $Z$ into skeleton and doomed particles. Define $\alpha:=1-f'(q)$, which is the factor by which the branching rate is reduced for the skeleton, giving an effective branching rate of $\beta \alpha$. It is easy to see that if $p_0>0$ and $\mu>1$, then $0<\alpha<1$ (see \cite[Lemma 4]{OC2013}).

\begin{theorem}[Survival in a  uniform field; $d\geq 2$ and $p_0>0$ ] \label{theorem3}
Let $p_0>0$ and $\mu>1$. Suppose that $\text{d}\nu/\text{d}x=v$, $v>0$. Then, for $d\geq 2$, $0<\widehat\varepsilon<1$ and any $n\in\mathbb{N}$,
\begin{align} &\underset{t\rightarrow\infty}{\lim}(\mathbb{E}\times P)\left(|Z^1((1-\widehat\varepsilon)t)|= 1 \mid \mathsf{S_t} \right)=1, \label{eq210} \\
& \underset{t\rightarrow\infty}{\lim}(\mathbb{E}\times P)\left(|Z^2((1-\widehat\varepsilon)t)|\leq \log^{(n)} t \mid \mathsf{S_t} \right)=1. \label{eq211}
\end{align}
\end{theorem}

\begin{proof} 

Let $(p^*_k)_{k\in\mathbb{N}_0}$ be the offspring probabilities for $Z^1$, i.e., the skeleton process. Then, (\ref{eq210}) follows from Theorem~\ref{theorem1}, since $p^*_0=0$, and one can adjust the branching rate of the skeleton (from $\beta$ to $\beta\alpha$) in order to make $p^*_1=0$.

To prove (\ref{eq211}), let $\widehat\varepsilon>0$, fix $n\in\mathbb{N}$, and define the events indexed by $t$ as 
\begin{equation}
K_t:=\left\{|Z^1((1-\widehat\varepsilon)t)|=1\right\}, \quad L_t:=\left\{|Z^2((1-\widehat\varepsilon)t)|\leq \log^{(n)} t\right\}. \nonumber
\end{equation}
Estimate
\begin{equation} (\mathbb{E}\times P)(L_t^c \mid \mathsf{S_t})\leq (\mathbb{E}\times P)(L_t^c \mid \mathsf{S_t},K_t)+(\mathbb{E}\times P)(K_t^c \mid \mathsf{S_t}) \label{eq24}.
\end{equation}
Note that the second term on the right-hand side of (\ref{eq24}) tends to zero by (\ref{eq210}). Now consider the first term. There is an obvious comparison between a BBM moving `freely' among the traps and a BBM conditioned to avoid the traps up to a certain time $t$. Conditioning on trap-avoiding tends to reduce the number of particles stochastically. Indeed, if we let $\Omega$ be the space of boundedly finite trap configurations, then an easy argument shows that for any $\omega\in\Omega$ and any $k\in\mathbb{N}$, 
\begin{equation} P^\omega(|Z(t)|\leq k\mid \mathsf{S_t})\geq P^\omega(|Z(t)|\leq k), \label{eq25}
\end{equation}
where $P^\omega$ is the law of the BBM conditioned to evolve in $\mathbb{R}^d$ with the trap configuration $\omega$ attached to it. Note that (\ref{eq25}) holds equally well if $Z$ is replaced by $Z^1$ or $Z^2$. Then, by (\ref{eq25}) and monotonicity of the integral, we may drop the conditioning on the first term on the right-hand side of (\ref{eq24}) and write
\begin{equation} (\mathbb{E}\times P)(L_t^c \mid \mathsf{S_t}, K_t)\leq P(L_t^c \mid K_t) \nonumber.
\end{equation}
It is clear that the presence of $K_t$ in the conditioning does not curb the validity of the inequality in (\ref{eq25}). Conditioned on $K_t$, there is exactly $1$ skeleton particle present at time $(1-\widehat\varepsilon)t$, which implies that there is exactly $1$ skeletal ancestral line up to that time. Hence, Lemma~\ref{lemma2} gives (\ref{eq211}).
\end{proof}

The proof of the following theorem is identical to that of the former; one only needs to replace $(1-\varepsilon)t$ by $(\eta^*-\varepsilon)t$ in the theorem statement and its proof.

\begin{theorem}[Survival in a radially decaying field; $d\geq 1$ and $p_0>0$] \label{theorem4}
Let $p_0>0$ and $\mu>1$. Let the trap intensity be radially decaying as in Theorem~\ref{theorem2}. For $n\in\mathbb{N}$ let $\log^{(n)} t$ be defined as before. Then for $d\geq 1$, $l>l_{cr}$, $0<\widehat\varepsilon<\eta^*$ and any $n\in\mathbb{N}$,
\begin{align} &\underset{t\rightarrow\infty}{\lim}(\mathbb{E}\times P)\left(|Z^1((\eta^*-\widehat\varepsilon)t)|=1\mid \mathsf{S_t}\right)=1,  \nonumber \\
& \underset{t\rightarrow\infty}{\lim}(\mathbb{E}\times P)\left(|Z^2((\eta^*-\widehat\varepsilon)t)|\leq \log^{(n)} t \mid \mathsf{S_t}\right)=1. \nonumber
\end{align}
\end{theorem}

\section{Corollaries: different types of optimal survival strategies}\label{section7}

In this section, using our results on the optimal survival strategies regarding the population size, namely Theorem~\ref{theorem3} and Theorem~\ref{theorem4}, we prove optimal survival results regarding the range of the BBM, and the size and position of the clearings in $\mathbb{R}^d$ as corollaries. Our proofs are in the same spirit as the ones for \cite[Thm.1.3(i)-(iv)]{E2003}. We emphasize that our results concerning the population size were all about suppressing the branching given survival among traps up to time $t$. Hence, the corollaries below arise in cases where there is some suppression of branching. For instance, when the trap intensity is uniform and $d=1$, in the case $l<l_{cr}$, the system does not need to suppress branching in order to avoid traps; hence this case is not studied below. Recall that $R=(R(t))_{t\geq 0}$ is the range process for the BBM.

\begin{corollary} [$d=1$]\label{cor1}
Let the trap intensity be uniform. If $Z$ is supercritical, then for $d=1$, $l>l_{cr}$ and $\varepsilon>0$,
\begin{align} &\underset{t\rightarrow\infty}{\lim}(\mathbb{E}\times P)\left(R(t)\subseteq B(0,\varepsilon t)\mid \mathsf{S_t} \right)=1, \label{eq27} \\
& \underset{t\rightarrow\infty}{\lim}(\mathbb{E}\times P)\left(B(0,\varepsilon t)\cap K\neq\emptyset \mid \mathsf{S_t} \right)=1. \label{eq28}
\end{align}
\end{corollary}
\noindent{\bf Note:} Regarding \eqref{eq28}, at the first sight, it may seem counterintuitive that trap avoidance implies the presence (and not the lack) of traps anywhere. However, for example in the $p_0=0$ case, the correct intuition is as follows: by Theorem \ref{theorem2}, given survival, the system only produces a single particle with overwhelming probability, and this single particle will most likely be close to the origin. Therefore, creating clearings further away from the origin would result in an unnecessary probabilistic cost.

\begin{proof} We prove the two displayed formulas separately.

{\bf (a) Proof of \eqref{eq27}:} Let $\varepsilon>0$ be fixed and let $|\mathcal{Z}(t)|=(|\mathcal{Z}^1(t)|,|\mathcal{Z}^2(t))|$ be the decomposition of the \textit{total progeny} for the BBM up to time $t$ for $t\geq 0$. Let $Z=(Z^1,Z^2)$ be the decomposition of $Z$ as before. From \cite[Thm.2.2]{OCE2017}, we know that $\eta^*=1$ when $d=1$. Let $0<\varepsilon'<1$ and $\delta>0$, which both will depend on $\varepsilon$ later. Define the events indexed by $t$ as 
\begin{equation}\mathcal{L}_t:=\left\{|\mathcal{Z}^2((1-\varepsilon')t)|\leq e^{\delta t}\right\},\quad K_t:=\left\{|Z^1((1-\varepsilon')t)|=1\right\}, \nonumber
\end{equation}
and
\begin{equation}
F_t:=\left\{R(t)\subseteq B(0,\varepsilon t)\right\}. \nonumber
\end{equation}
It is enough to show that
\begin{equation} \underset{t\rightarrow\infty}{\limsup}\,\frac{1}{t}\log (\mathbb{E}\times P)(F_t^c\cap \mathsf{S_t})<-I, \nonumber
\end{equation}
where $I=\beta \alpha$ (see \cite[Thm.1]{O2016}). Estimate
\begin{equation} (\mathbb{E}\times P)(F_t^c \cap \mathsf{S_t})\leq P(F_t^c\cap \mathcal{L}_t\cap K_t)+(\mathbb{E}\times P)(K_t^c \cap \mathsf{S_t} )+P(\mathcal{L}_t^c \cap K_t) \label{eq29}.
\end{equation}
The second term on the right-hand side of \eqref{eq29} is lower order than $\exp[-It]$ on an exponential scale as $t\rightarrow\infty$ by the proof of Theorem~\ref{theorem1}. The third term can be written as $P(\mathcal{L}_t^c \mid K_t)P(K_t)$. Similarly to the argument leading to \eqref{eq3}-\eqref{eq5}, one can show that $P(K_t)=\exp[-\beta\alpha t+o(t)]$ since the effective branching rate for the skeleton is $\beta\alpha$. Now consider $P(\mathcal{L}_t^c \mid K_t)$. Conditioned on $K_t$, since there is only one skeleton particle present at time $(1-\varepsilon')t$ and the expected number of occurrences of branching along its skeletal line is $\beta (1-\varepsilon') t$ up to time $(1-\varepsilon')t$, we have 
\[E[\,|\mathcal{Z}^2((1-\varepsilon')t)|\mid K_t]=\kappa \beta (1-\varepsilon') t,  \]   
where $\kappa>0$ is the product of the expected total progeny of a doomed subtree and the expected doomed offspring of a skeleton particle, which are both finite and don't depend on $t$. Then, Markov inequality implies that $P(\mathcal{L}_t^c \mid K_t)\leq \exp[-\delta t+o(t)]$ so that $P(\mathcal{L}_t^c \mid K_t)P(K_t)$ is lower order than $\exp[-It]$ on an exponential scale. It remains to show that 
\begin{equation} \underset{t\rightarrow\infty}{\limsup}\,\frac{1}{t}\log P(F_t^c\cap \mathcal{L}_t\cap K_t)<-I. \nonumber
\end{equation}
Define the following events: 
\begin{align} 
F_t^1:=&\left\{R((1-\varepsilon')t)\subseteq B(0,\varepsilon t/2)\right\}, \nonumber \\
F_t^2:=&\bigr\{\text{each sub-BBM emanating from one of the `parent' particles at} \nonumber \\ 
&\ \text{time $(1-\varepsilon')t$ is contained in an $\varepsilon t/2$-ball around the position} \nonumber \\
&\ \text{of the parent particle.} \bigl\} \nonumber
\end{align}
It is clear that $F_t^1 \cap F_t^2\subseteq F_t$. Therefore, using de Morgan's law, followed by the union bound, it suffices to show the following two inequalities:
\begin{align} &\underset{t\rightarrow\infty}{\limsup}\,\frac{1}{t}\log P((F_t^1)^c\cap \mathcal{L}_t\cap K_t)<-I, \label{eq31} \\
&\underset{t\rightarrow\infty}{\limsup}\,\frac{1}{t}\,\log P((F_t^2)^c\cap \mathcal{L}_t\cap K_t)<-I. \label{eq32}
\end{align}
On the event $(F_t^1)^c\cap \mathcal{L}_t\cap K_t$, the following probabilistic costs arise: The system has only $1$ skeleton particle throughout the time interval $[0,(1-\varepsilon')t]$, which has probability $\exp[-\beta \alpha (1-\varepsilon')t]$. Also, at least one Brownian path must go outside $B(0,\varepsilon t/2)$ for some $s\in[0,(1-\varepsilon')t]$, which has probability at most $\exp\left[-\varepsilon^2/[8(1-\varepsilon')t]+\delta t+o(t)\right]$ by \cite[Lemma 5]{OCE2017} and the union bound since on the event $\mathcal{L}_t\cap K_t$, there are at most $\exp[\delta t+o(t)]$ particles in the system at all times in the period $[0,(1-\varepsilon')t]$. Therefore, by independence of branching and motion mechanisms, we obtain
\begin{equation} \underset{t\rightarrow\infty}{\limsup}\,\frac{1}{t}\log P((F_t^1)^c\cap \mathcal{L}_t\cap K_t)\leq -\beta \alpha (1-\varepsilon')-\frac{\varepsilon^2}{8(1-\varepsilon')}+\delta. \nonumber
\end{equation} 
Then, since $I=\beta\alpha$ when $d=1$, $l>l_{cr}$; to prove \eqref{eq31}, it suffices to choose $\varepsilon'>0$ such that the inequality
\begin{equation} \beta \alpha (1-\varepsilon')+\frac{\varepsilon^2}{8(1-\varepsilon')}>\beta\alpha+\delta \label{eq33}
\end{equation}
is satisfied.

Now consider the event $(F_t^2)^c\cap \mathcal{L}_t\cap K_t$. On the event $(F_t^2)^c$, at least one sub-BBM emanating from one of the `parent' particles at time $(1-\varepsilon')t$ must escape its $\varepsilon t/2$-ball around the position of the parent particle. Fix one such sub-BBM. By the proof of Proposition 1 in \cite{OCE2017}, an argument similar to the one leading to \eqref{eq7} shows that if 
\begin{equation}
\varepsilon/2>2\varepsilon'\sqrt{2\beta m} \label{eq34},
\end{equation}
then the probability that this sub-BBM exits a $\varepsilon t/2$-ball around the position of the parent particle in the remaining time $\varepsilon't$ is at most $\exp[-3\beta m \varepsilon' t+o(t)]$. Since $\mathcal{L}_t\cap K_t$ implies the existence of at most $\exp[\delta t+o(t)]$ many particles at time $(1-\varepsilon')t$, this introduces a factor of at most $\delta t$ to the exponent in the latter estimate. Again, by independence of branching and motion, we obtain
\begin{equation} \underset{t\rightarrow\infty}{\limsup}\,\frac{1}{t}\log P((F_t^2)^c\cap \mathcal{L}_t\cap K_t)\leq -\beta \alpha (1-\varepsilon')-3\beta m \varepsilon'+\delta< -\beta\alpha,  \nonumber
\end{equation} 
provided that $\delta$ is small enough, where the last inequality follows since $3m>\alpha$. (Recall that the BBM is supercritical, which means $m>1$, whereas $\alpha\leq 1$.) Finally, to satisfy \eqref{eq33} and \eqref{eq34}, and hence to complete the proof of \eqref{eq27}, choose $\delta$ small enough and $\varepsilon'=\min\left\{\varepsilon^2/(8\beta \alpha),\varepsilon/(4\sqrt{2\beta m})\right\}$.

\medskip
{\bf (b) Proof of \eqref{eq28}:} Let $0<\varepsilon'<\varepsilon$, and define the events indexed by $t$ as 
\begin{equation}
D_t:=\left\{R(t)\subseteq B(0,\varepsilon' t)\right\}, \quad G_t^1:=\left\{B(0,\varepsilon t)\cap K\neq\emptyset\right\}, \quad G_t^2:=\left\{B(0,\varepsilon' t+r)\cap K\neq\emptyset\right\}. \nonumber
\end{equation}
(Recall that $r$ is the constant trap radius.) It is clear that $(G_t^1)^c\subset(G_t^2)^c$, and by the definition of Poisson random measure, the probabilities of $(G_t^2)^c$ and $(G_t^1)^c$ differ by $\varepsilon-\varepsilon'$ on an exponential scale. Estimate
\begin{equation} (\mathbb{E}\times P)((G_t^1)^c\cap\mathsf{S_t})\leq(\mathbb{E}\times P)((G_t^1)^c\cap D_t)+(\mathbb{E}\times P)(\mathsf{S_t}\cap D_t^c). \label{eq36}
\end{equation} 
The second term on the right-hand side of \eqref{eq36} is lower order than $\exp[-It]$ on an exponential scale as $t\rightarrow\infty$, since it was shown previously that each term on the right-hand side of \eqref{eq29} is such. The first term has the following asymptotics:
\begin{align} \underset{t\rightarrow\infty}{\limsup}\,\frac{1}{t}\log(\mathbb{E}\times P)((G_t^1)^c\cap D_t)=&\:\underset{t\rightarrow\infty}{\limsup}\,\frac{1}{t}\log\left[\mathbb{P}(G_t^1)^c P(D_t)\right] \nonumber \\
<&\:\underset{t\rightarrow\infty}{\limsup}\,\frac{1}{t}\log\left[\mathbb{P}(G_t^2)^c P(D_t)\right] \nonumber \\
=&\:\underset{t\rightarrow\infty}{\limsup}\,\frac{1}{t}\log(\mathbb{E}\times P)((G_t^2)^c\cap D_t)\leq -I, \nonumber
\end{align} 
where the first inequality follows from the fact that the probabilities of $(G_t^2)^c$ and $(G_t^1)^c$ differ by $\varepsilon-\varepsilon'$ on an exponential scale, and the last inequality follows since $(G_t^2)^c\cap D_t\subseteq\mathsf{S_t}$. This completes the proof of \eqref{eq28}.  
\end{proof}

\begin{corollary}[$d\ge 2$]\label{cor2}
Let the trap intensity be uniform. If $Z$ is supercritical, then for $d\geq 2$ and $\varepsilon>0$,
\begin{align} 
&\underset{t\rightarrow\infty}{\lim}(\mathbb{E}\times P)\left(R(t)\subseteq B(0,\varepsilon t)\mid \mathsf{S_t} \right)=1, \label{eq281} \\
& \underset{t\rightarrow\infty}{\lim}(\mathbb{E}\times P)\left(B(0,\varepsilon t^{1/d})\cap K\neq\emptyset \mid \mathsf{S_t}\right)=1. \label{eq282}
\end{align}
\end{corollary}

\begin{proof} For the proof of \eqref{eq281}, refer to the proof of Corollary~\ref{cor1}. 

To prove \eqref{eq282}, let $0<\varepsilon'<\varepsilon$, and define the events indexed by $t$ as 
\begin{equation}
D_t:=\left\{R(t)\subseteq B(0,\varepsilon' t)\right\}, \quad A_t:=\left\{B(0,\varepsilon t^{1/d})\cap K\neq\emptyset\right\}. \nonumber
\end{equation} 
Estimate
\begin{equation} (\mathbb{E}\times P)(A_t^c\cap\mathsf{S_t})\leq(\mathbb{E}\times P)(A_t^c\cap D_t)+(\mathbb{E}\times P)(\mathsf{S_t}\cap D_t^c)=:\mathbf {I}+\mathbf {II}. \label{eq284}
\end{equation} 
Now, $\mathbf {II}=o((\mathbb{E}\times P)(\mathsf{S_t}))$ as $t\rightarrow\infty$ by \eqref{eq281}. By the independence of the BBM and the Poisson random measure, 
\begin{equation} \mathbf {I}=\mathbb{P}(A_t^c) P(D_t) 
= \exp\left(-v\omega_d \varepsilon^d t\right) \exp\left[-\left(\beta\alpha-\sqrt{\beta\alpha/(2m)}\varepsilon'\right)t+o(t)\right], \label{eq285}
\end{equation}
where $\omega_d$ is the volume of the $d$-dimensional unit ball and $v>0$ is the constant trap intensity. In passing to the second equality of \eqref{eq285}, we have used the definition of Poisson random measure and \cite[Thm.2]{E2004}. Recall that $\beta\alpha$ is the branching rate for the skeleton, where $0<\alpha\leq 1$ if the BBM is supercritical. Finally, since $I=\beta\alpha$ when $d\geq 2$ (see \cite[Thm.1]{O2016}), to complete the proof, choose $\varepsilon'>0$ sufficiently small  to satisfy the inequality
\begin{equation} v\omega_d \varepsilon^d+\beta\alpha-\sqrt{\beta\alpha/(2m)}\varepsilon'>\beta\alpha, \nonumber\end{equation}
that is,
\begin{equation}  \nonumber
\varepsilon' < \frac{v\omega_d \varepsilon^d}{\sqrt{\beta\alpha/(2m)}},
\end{equation}
and then $\mathbf {I}$ in \eqref{eq284} also satisfies $\mathbf {I}=o((\mathbb{E}\times P)(\mathsf{S_t}))$ 
as $t\rightarrow\infty$.
\end{proof}

\bibliographystyle{plain}

\end{document}